\documentclass[]{interact}

\usepackage{color}
\usepackage{amsmath}
\usepackage{amssymb}
\usepackage{times}

\usepackage{graphicx}

\usepackage{epstopdf}
\usepackage[caption=false]{subfig}

\usepackage[numbers,sort&compress]{natbib}
\bibpunct[, ]{[}{]}{,}{n}{,}{,}
\makeatletter
\def\NAT@def@citea{\def\@citea{\NAT@separator}}
\makeatother

\theoremstyle{plain}
\newtheorem{theorem}{Theorem}[section]
\newtheorem{lemma}[theorem]{Lemma}
\newtheorem{assumption}[theorem]{Assumption}
\newtheorem{corrolary}[theorem]{Corollary}

\theoremstyle{definition}
\newtheorem{definition}[theorem]{Definition}

\theoremstyle{remark}
\newtheorem{remark}{Remark}

\newcommand{\rset}{\mathbb{R}}

\usecounter{algorithmenumi}

\providecommand{\norm}[1]{\lVert#1\rVert}

\begin{document}
	
	
	\title{New nonasymptotic convergence rates of stochastic proximal point algorithm for stochastic convex optimization}

	\author{
		\name{Andrei P\u atra\c scu  \thanks{CONTACT Andrei P\u atra\c scu. Email: andrei.patrascu@fmi.unibuc.ro}} 
		\affil{ University Bucharest, Str. Academiei 14, 010014, Bucharest, Romania }
	}
	
	\maketitle
	

\begin{abstract}

\noindent Large sectors of the recent optimization literature focused in the last decade on the development of optimal stochastic first order schemes for constrained convex models under progressively relaxed assumptions. Stochastic proximal point is an iterative scheme born from the adaptation of proximal point algorithm to noisy stochastic optimization, with a resulting iteration related to stochastic alternating projections.
Inspired by the scalability of alternating projection methods,  we start from the (linear) regularity assumption, typically used in convex feasiblity problems to guarantee the linear convergence of stochastic alternating projection methods, and analyze a general weak linear regularity condition which facilitates convergence rate boosts in stochastic proximal point schemes. Our applications include many non-strongly convex functions classes often used in machine learning and statistics.
Moreover, under weak linear regularity assumption we guarantee  $\mathcal{O}\left(\frac{1}{k}\right)$ convergence rate for SPP, in terms of the distance to the optimal set, using only projections onto a simple component set. Linear convergence is obtained for interpolation setting, when the optimal set of the expected cost is included into the optimal sets of each functional component.
\end{abstract}

\begin{keywords}%
Stochastic proximal point, stochastic alternating projections, quadratic growth, nonsmooth optimization, linear convergence, sublinear convergence rate
\end{keywords}

\section{Introduction}

\noindent Many large-scale modern applications \cite{ChoBar:04,Her:09,HerChe:08} are often modeled by a convex feasiblity problem (CFP):
\begin{equation*}
 \text{find} \quad x \in \bigcap_{\xi \in \Omega} X_{\xi}
\end{equation*}
where $\{X_{\xi}\}_{\xi \in \Omega}$ are simple closed convex sets. 
There exist plenty of iterative algorithms which efficiently solve CFPs under various regularity conditions on the feasible sets, from which we mention only the stochastic alternating projections (SAP) schemes due to their relevance to our paper, see \cite{GubPol:67,Ned:10,CenChe:12,NecRic:18}. Even in their simplest form, based on individual projections onto randomly chosen simple sets $X_{\xi}$,
stochastic alternating projections attain linear convergence and high performance due to their computationally cheap iteration which do not scale with the number of sets \cite{Ned:10,CenChe:12,NecRic:18}.

\noindent It has been shown in \cite{NecRic:18} that any CFP can be put in the following form:
\begin{equation*}
\min\limits_{x \in \rset^n} \;\;  \mathbb{E}_{\xi \in \Omega}[\mathbb{I}_{X_{\xi}}(x)],
\end{equation*}
where $\mathbb{I}_{X_{\xi}}(x)$
is the indicator function of set $X_{\xi}$, defined by values $0$ for $x \in X_{\xi}$ and $+\infty$ otherwise, and the expectation operator $\mathbb{E}$  over a uniform distributed random variable $\xi$. Stochastic alternating projections samples randomly an index $\xi_k$ and generate the sequence:
\begin{align*}
x^{k+1} :=\arg\min_z \; \mathbb{I}_{X_{\xi_k}}(z) + \frac{1}{2}\norm{z-x^k}^2 \quad \left(= \pi_{X_{\xi_k}}(x^k)\right).
\end{align*}
Sublinear rates are easily obtained for mild convex sets \cite{GubPol:67,Ned:10,NecRic:18}. For faster convergence, SAP requires better conditioned structure. For instance, 
it has been shown in \cite{Ned:10,NecRic:18} that under \emph{linear regularity property}, characteristic in particular to polyhedral sets,  SAP exhibits linear convergence. Inspired by these projection  favorable landscapes, in this paper we aim to analyze further extensions of regularity properties to general convex functions, and approach the following stochastic optimization problem:
\begin{align}
\label{problem_intro}
\min\limits_{x \in \rset^n} & \;\;  F(x) := \left(\mathbb{E} [F(x;\xi)]\right),
\end{align}
where each component $F(\cdot, \xi): \rset^n \mapsto (-\infty, + \infty]$ is proper convex and lower-semicontinuous. The natural extension of the proximal iteration of an indicator function, involved in the SAP algorithm, towards more general proximal operators of convex functions leads to the stochastic proximal point (SPP) scheme for \eqref{problem_intro} \cite{RyuBoy:16,TouTra:16,PatNec:18}. 
Thus, given the smoothing parameter sequence $\{\mu_k\}_{k \ge 0}$, the vanilla SPP iteration has the form 
\begin{align*}
 x^{k+1} : = \arg\min\limits_{z} F(z;\xi_k) + \frac{1}{2\mu_k}\norm{z - x^k}^2 ,
\end{align*}
Recently its convergence behavior has been analyzed under various assumptions and several advantages have been theoretically and empirically illustrated over SGD algorithms \cite{Bia:16,RyuBoy:16,TouTra:16,PatNec:18}. However, inspired by two facts: the boost from sublinear to linear rate of SAP under linear regularity and by the connection between SPP and SAP, we change the perspective adopted in previous references and address the following natural questions:

\vspace{3pt}

\noindent \emph{
Does the generalization of linear regularity guarantee convergence rate boost of SPP? Which practical models satisfy this generalized regularity?}

\vspace{3pt}

\noindent  In our paper we answer these questions by finding that a somehow straight generalization of  linear regularity improves the iteration complexity of SPP by one order of magnitude, while particularly maintaining linear convergence for linearly regular CFPs. The main contributions of this paper are:

\noindent $(i)$ We offer a unified theoretical perspective over SPP and SAP schemes, using a unified convergence rate analysis based on simple novel arguments. Up to our knowledge, this is the first unified analysis providing complexity results for stochastic proximal point and stochastic alternating projections. 

\noindent $(ii)$ We provide sublinear/linear convergence rates for SPP scheme on constrained convex optimization. The key structural assumption allowing this general result is the weak linear regularity property, a natural generalization of classical linear regularity of convex sets.

\noindent $(iii)$ Our analysis applies to constrained optimization with complicated constraints (see Section \ref{sec_funcclass}). In our analysis SPP requires only simple projections onto individual sets, unlike most projected stochastic first order schemes that use computationally expensive full projections onto the entire feasible set. 

\noindent $(iv)$ The new proof techniques based on the weak linear regularity property are simpler than previous approaches. We show a sublinear $\mathcal{O}\left( \frac{1}{k} \right)$ convergence rate  for the stochastic proximal point algorithm in terms of the distance from the optimal set. Moreover, in the interpolation case when the functional components share minimizers, linear convergence is obtained.

\subsection{Related work}
Significant parts of the tremendous literature on stochastic optimization algorithms focused on the theoretical and practical behavior of stochastic first order schemes under different convexity properties, 
see \cite{RamSin:13, YanLin:18,    Bia:16,NemJud:09,MouBac:11,RakSha:12,HazSat:14,ShaSin:11,NguNgu:18,RosVil:14,XuLin:17}. Due to its simplicity, the traditional method of choice for most supervised machine learning problems is the SGD algorithm. 
On one hand, most works on constrained stochastic optimization problem $\min\limits_{x \in X} F(x) \left(:= \mathbb{E}[f(x;\xi)]\right) $ develop SGD-type schemes involving projections onto feasible sets, which could bring   a significant computational burden for complicated sets. 
On the other hand, a great proportion of previous work assume bounded gradients of the objective function $F$, which do not hold for smooth quadratically growing functions such as the linear regression cost.
However, our analysis of the SPP scheme alows us to overcome these limitations. 

\noindent The stochastic proximal point algorithm has been recently analyzed using various differentiability assumptions, see \cite{TouTra:16,RyuBoy:16,Bia:16,PatNec:18,KosNed:13,WanBer:16}.
In \cite{TouTra:16} is considered the typical  stochastic learning model involving the expectation of random particular components $f(x;\xi)$ defined by the composition of a smooth function and a linear operator, i.e.: 
$  f(x;\xi) = \ell(a^T_{\xi} x), $
where $a_\xi \in \rset^n$.  The complexity analysis requires the linear composition form, i.e. $\ell(a^T_{\xi} x)$, and that the objective function $\mathbb{E}[\ell(a^T_{\xi} x)]$ to be smooth and strongly convex. The nonasymptotic convergence of the
SPP with decreasing stepsize $\mu_k=\frac{\mu_0}{k^{\gamma}}$, with
$\gamma \in (1/2, 1]$, has been analyzed in the quadratic mean and
an $\mathcal{O}\left(\frac{1}{k^{\gamma}} \right)$  convergence rate has been derived.
The generalization of these convergence guarantees is undertaken in \cite{PatNec:18}, where no linear composition structure is required and an (in)finite number of constraints are included in the stochastic model. However, the stochastic model from \cite{PatNec:18} requires strong convexity and Lipschitz gradient continuity for each functional component $f(\cdot;\xi)$. Furthermore, it is explicitly specified that their analysis do not extend to certain models, such as those with nonsmooth functional components $\hat{f}(x;\xi) := f(x;\xi) + \mathbb{I}_{X_{\xi}}(x)$, where $f(\cdot;\xi)$ is smooth and convex. Note that our analysis surpasses these restrictions and provides a natural generalization of \cite{PatNec:18} to nonsmooth constrained models.

\noindent In  \cite{RyuBoy:16}, the SPP scheme with decreasing stepsize $\mu_k=\frac{\mu_0}{k}$ has been applied to problems with the objective function having Lipschitz continuous gradient and the restricted strong convexity property, and its asymptotic global convergence  is derived. A sublinear $\mathcal{O}\left( \frac{1}{k}\right)$ asymptotic  convergence rate in the quadratic mean has been given.
In this paper we make more general assumptions on the objective function, which hold for restricted strongly convex functions, and provide nonasymptotic convergence analysis of the SPP for a more general stepsize $\mu_k=\frac{\mu_0}{k^\gamma}$, with $\gamma >0$. Further, in  \cite{Bia:16} a  general asymptotic convergence analysis of slightly modified SPP scheme has been provided, under mild convexity assumptions on a finitely constrained stochastic problem. 
Although this scheme is very similar to the  SPP algorithm, only the almost sure asymptotic convergence has been provided in  \cite{Bia:16}. 

\noindent Recently, in \cite{AsiDuc:18}, the authors analyze SPP schemes for shared minimizers stochastic optimization obtaining linear convergence results, for variable stepsize SPP. Also they obtain $\mathcal{O}\left( \frac{1}{k^{1/2}} \right)$ for convex Lipschitz continuous objectives and, furthermore, $\mathcal{O}\left( \frac{1}{k} \right)$ for strongly convex functions. Remarkably, they eliminate any continuity assumption for the sublinear rate in the strongly convex case, which allows indicator functions. However, our analysis uses non-trivially the linear regularity of feasible set for obtaining better convergence constants. Moreover, we use quadratic growth relaxations of strong convexity assumption which allow a unified treatment of SPP and AP schemes.


\noindent \textbf{Notations}. We use notation $[m] =\{1,\cdots, m\}$. For $x,y \in \rset^n$ denote the scalar product  $\langle x,y \rangle = x^T y$ and Euclidean norm by $\|x\|=\sqrt{x^T x}$. The projection operator onto set $X$ is denoted by $\pi_X$ and the distance from $x$ to the set $X$ is denoted $\text{dist}_X(x) = \min_{z \in X} \norm{x-z}$. 
For function $f(\cdot;\xi)$, $\text{dom}(f)$ is the effective domain and we use notations $\partial f(x;\xi)$ for the subdifferential set at $x$  and $g_f(x;\xi)$ for a subgradient of $f$ at $x$. If $f(\cdot;\xi)$ is differentiable we use the gradient notation $\nabla f(\cdot;\xi)$. Also we use $g_{X}(x,\xi) \in \mathcal{N}_{X_{\xi}}(x)$ for a subgradient of $\mathbb{I}_{X_{\xi}}(x)$. Finally, we  define the function $\varphi_{\alpha}: (0, \infty) \to
\rset^{}$ as:
$  \varphi_{\alpha}(x) =
\begin{cases}
(x^{\alpha} - 1)/\alpha, & \text{if} \; \alpha \neq 0 \\
\log(x), & \text{if} \; \alpha = 0.
\end{cases}$


\subsection{Problem formulation}

\noindent  We consider the following main stochastic minimization:
\begin{align}
\label{problem}
F^* = & \min_{x \in \rset^n}  F(x) \;\; \left( := \mathbb{E}[F(x ; \xi) ] \right) 
\end{align}
where $F(\cdot;\xi): \rset^n \to (-\infty, + \infty] $ are proper convex  and lower-semicontinuous functions. 
The random variable $\xi$ has its associated probability space $(\Omega, \mathbb{P})$. 
When the functional components include indicator functions then \eqref{problem} covers constrained models :
\begin{align*}
\min\limits_{x \in \rset^n}  \;  \mathbb{E}_{\xi \in \Omega_1}[f(x;\xi)] \qquad \text{s.t.} \; \bigcap_{\xi \in \Omega_2} X_{\xi},
\end{align*}
where $\Omega = \Omega_1 \cup \Omega_2$.
Denote the optimal set with $X^*$ and $x^*$ some optimal point for \eqref{problem}.
\begin{assumption} \label{assump_basic}
Assume that the central problem \eqref{problem} satisfies:

\noindent $(i)$ The optimal set $X^*$ is nonempty.

\noindent $(ii)$ There exists subgradient mapping $g_F: \rset^n \times \Omega \mapsto \rset^n$ such that $g_F(x;\xi) \in \partial F(x;\xi), \forall \xi \in \Omega$ and $\mathbb{E}[g_F(x;\xi)] \in \partial F(x).$ 

\noindent $(iii)$  $F(\cdot;\xi)$ has bounded gradients on the optimal set: there exists $\mathcal{S}^*_F \ge 0$ such that $\mathbb{E}\left[\norm{g_F(x^*;\xi)}^2\right]  \le \mathcal{S}^*_F  < \infty$ for all $x^* \in X^*$;

\noindent $(iv)$ For any $g_F(x^*) \in \partial F(x^*)$ there exists bounded subgradients $g_F(x^*;\xi) \in \partial F(x^*;\xi),$ such that $ \mathbb{E}[g_F(x^*;\xi)] = g_F(x^*)$ and $\mathbb{E}[\norm{g_F(x^*;\xi)}^2]<\mathcal{S}^*_F$. Moreover, for simplicity we assume throughout the paper $g_F(x^*) = \mathbb{E}[g_F(x^*;\xi)] = 0$
\end{assumption}

\noindent The first part of the above assumption is natural in the stochastic optimization problems. The Assumption $(ii)$ guarantee the existence of a subgradient mapping. The third part Assumption $(iii)$ is standard in the literature related to stochastic algorithms.

\begin{remark}
The assumption $(iv)$ needs a more consistent discussion. Denote $\mathbb{E}[\partial F(\cdot;\xi)] = \left\{\mathbb{E}[g_F(\cdot;\xi)] \;|\; g_F(\cdot;\xi) \in \partial F(\cdot;\xi)  \right\}$. In general, for convex functions it can be easily shown $\mathbb{E}\left[ \partial F(x;\xi)\right] \subseteq \partial F(x)$ for all $x \in \text{dom}(F)$ (see \cite{RocWet:82}). However, $(iv)$ is guaranteed by the stronger equality
\begin{align}\label{diff_equality}
\mathbb{E}\left[ \partial F(x;\xi)\right] = \partial F(x).
\end{align} 

\noindent \textbf{Discrete case.} Let us consider finite discrete domains $\Omega = \{1, \cdots, m\}$. Then \cite[Theorem 23.8]{Roc:70} guarantees that the finite sum  objective of \eqref{problem_intro} satisfy \eqref{diff_equality} if $ \bigcap\limits_{\xi \in \Omega}\text{ri}(\text{dom}(f(\cdot;\xi))) \neq \emptyset$.
The $\text{ri}(\text{dom}(\cdot))$ can be relaxed to $\text{dom}(\cdot)$ for polyhedral components. 
In particular,  let $X_1, \cdots, X_m$ be finitely many closed convex satisfying qualification condition: $\bigcap\limits_{i=1}^m \text{ri}(X_i) \neq \emptyset$, then also \eqref{diff_equality} holds, i.e. $\mathcal{N}_{X}(x) = \sum\limits_{i=1}^m \mathcal{N}_{X_i}(x)$ (see ( by \cite[Corrolary 23.8.1]{Roc:70})). Again, $\text{ri}(X_i)$ can be relaxed to the set itself for polyhedral sets. 
As pointed by \cite{BauBor:99}, the (bounded) linear regularity property of $\{X_i\}_{i=1}^m$ implies the intersection qualification condition. 

Under support of these arguments, observe that $(iv)$ can be easily checked for our finite-sum examples given below in the rest of our sections.

\noindent \textbf{Continuous case.}
In the nondiscrete case, sufficient conditions for \eqref{diff_equality} are discussed in \cite{RocWet:82}. Based on the arguments from \cite{RocWet:82}, an assumption equivalent to $(iv)$ is considered in \cite{SalBia:17} under the name of $2-$integrable representation of $x^*$ (definition in \cite[Section B]{SalBia:17}). On short, if $g(\cdot;\xi)$ is normal convex integrand with full domain then $x^*$  admits an $2-$integrable representation $\{g_F(x^*;\xi)\}_{\xi \in \Omega}$, and implicitly $(iv)$ holds. 

Lastly, we mention that deriving a more complicated result similar to Lemma \ref{lemma_auxres} we could avoid assumption $(iv)$. However, since $(iv)$ facilitates the simplicity and naturality of our results and while our target applications are not excluded, we assume throughout the paper that $(iv)$ holds. 
\end{remark}
\noindent We use  the approximation of the functions $F(\cdot;\xi)$ given by their Moreau envelope, that is:
$ F_{\mu} (x;\xi) := \min\limits_{z \in \rset^n}  \;  F(z;\xi) + \frac{1}{2\mu} \norm{z - x}^2$
with smoothing parameter $\mu>0$. 
We denote the proximal operator corresponding to $F(\cdot;\xi)$ with:
\begin{align*}
z_{\mu}(x;\xi) &= \arg\min\limits_{z \in \rset^n} F(z;\xi) + \frac{1}{2\mu} \norm{z - x}^2.
\end{align*}
In particular, when $F(x;\xi) = \mathbb{I}_{X_{\xi}}(x)$ the proximal operator becomes the projection operator $z_{\mu}(x;\xi) = \pi_{X_{\xi}}(x)$. 
The approximate $F_{\mu}(\cdot;\xi)$ has Lipschitz continuous gradient with constant $\frac{1}{\mu}$ and preserves the convexity properties of $F(\cdot;\xi)$, see \cite{RocWet:98}. Hence, it results a new model:
\begin{align}\label{moreau_problem}
F^*_{\mu} = \min\limits_{x \in \rset^n} \; F_{\mu}(x) \quad \left(:= \mathbb{E}[F_{\mu}(x;\xi)]\right).
\end{align}
We denote $X^*_{\mu} = \arg\min\limits_{x} F_{\mu}(x)$, but in general $X^*_{\mu} \neq X^*$. However, there are particular contexts when \eqref{moreau_problem} is equivalent with \eqref{problem}.  As we pointed earlier, in the CFPs, for finite $\Omega$:
$ F(x;\xi) = \mathbb{I}_{X_{\xi}}(x) \;  \text{and} \; X^* = \bigcap\limits_{\xi \in \Omega} X_{\xi}.$
Here,
$F_{\mu}(x;\xi) = \text{dist}_{X_{\xi}}^2(x)$ and the objective of \eqref{moreau_problem} reduces to:
$ F_{\mu}(x) = \frac{1}{2\mu}\mathbb{E}[\text{dist}_{X_{\xi}}^2(x)].$
When $ X^* \neq \emptyset$ then $X^*_{\mu} = X^*$ for all $\mu > 0$.


\section{Weak linear regularity property}


In the CFPs framework, it is widely known that linear regularity property enhances linear convergence of the projection methods (see \cite{Ned:10,CenChe:12,NecRic:18}).
\begin{definition}
	Let $\{X_{\xi}\}_{\xi \in \Omega}$ be convex sets with nonempty intersection $X = \bigcap\limits_{\xi \in \Omega} X_{\xi}$. They are linearly regular with constant $\sigma_X > 0$ if: 
	\begin{align}\label{lreg_definition}
	\sigma_X \text{dist}_{X}^2(x) \le   \mathbb{E}\left[\text{dist}_{X_{\xi}}^2(x)\right]    \qquad \forall x \in \rset^n.
	\end{align}
\end{definition}

\noindent Furthermore, convergence rates for minibatch stochastic projection methods were derived in \cite{NecRic:18}, which depends on the minibatch size under the linear regularity property. 

\noindent Inspired by the fact that the powerful linear regularity ensures linear convergence in SAP, a straight generalization of linear regularity from CFPs world to stochastic optimization \eqref{problem} should further facilitate superior sublinear and linear convergence orders of the stochastic first-order methods.

\noindent Notice that \cite{PatNec:18} analyzed SPP under linearly regular constraints, but they do not exploit eventual interpolation property of the objective function. For example, for the linear regression $\min\limits_x \; \mathbb{E}[(A_{\xi}x - b_{\xi})^2]$, under existence of solution $Ax = b$, the analysis \cite{PatNec:18} provides $\frac{1}{\sqrt{k}}$ convergence, while from our analysis yields \textit{linear} convergence rate for this type of problems. 

\noindent We further state the weak linear regularity assumption which generalizes the classical linear regularity to our model \eqref{problem}.

\begin{assumption}\label{assump_localgrowth}
	The objective function $F$ satisfies weak linear regularity property if  there exists constants $\sigma_{F,\mu}>0$ and $\beta > 0$ such that for any $\mu > 0$:
	\begin{equation}\label{proxqg}
	\frac{\sigma_{F,\mu}}{2}\text{dist}_{X^*}^2(x) \le F_{\mu}(x) - F^*_{\mu}  + \mu \beta \qquad \forall x \in \rset^n.
	\end{equation}
	Moreover, the mapping $\mu \mapsto \sigma_{F,\mu}$ is nonincreasing in $\mu$.	
\end{assumption}
\noindent In CFP case ($f=0,\mu=1$), the right hand side of \eqref{lreg_definition} is $\mathbb{E}[\text{dist}^2_{X_{\xi}}(x)] = F_1(x)$ and thus the linear regularity is covered by constants $\sigma_{F,1} = 2\sigma_X$ and $\beta=0$.

\noindent Also the assumption \ref{assump_localgrowth} can be interpreted as a generalized quadratic growth since for $\mu = 0$  reduces to the well-known pure quadratic growth property for the objective function $F$, which has been extensively analyzed in the deterministic setting, see for example \cite{YanLin:18, XuLin:17, NecNes:16,CenChe:12}. Since in many practical applications the strong convexity does not hold, first-order algorithms exhibit linear convergence under pure quadratic growth and certain additional smoothness conditions \cite{NecNes:16}. However, for general stochastic  first order algorithms the geometric convergence feature cannot be attained, due to the variance in the chosen stochastic direction. Sublinear convergence of the restarted SGD has been shown in \cite{XuLin:17, YanLin:18} under the quadratic growth property and bounded gradients. 

We show in section \ref{sec_funcclass} that weak linear regularity holds for the following particular prediction models:

\begin{itemize}
	\item[$(I)$] Constrained Linear Regression \cite{StoIro:19}: let $X_{\xi} = \{z \in \rset^n: c_{\xi}z \le d_{\xi}\}$
	 $$\min\limits_x \;\; \mathbb{E}[(a_{\xi}x-b_{\xi})^2] + \mathbb{E}[\mathbb{I}_{X_{\xi}}(x)]$$ 
	\item[$(II)$] Regularized dual problem of Multiple kernel learning [2]: let $(X,y)$ be a sampled labeled dataset and $C$ a penalty parameter then dual MKL
	can be formulated as
	\begin{align*}
	\min\limits_{\gamma \in \rset^{}, \alpha \in \rset^n} &  \gamma^2 + \frac{\lambda}{2}\norm{\alpha}^2_2 - \alpha^T e + \mathbb{E}_{\xi \in [m+2]}\left[ \mathbb{I}_{X_{\xi}}(\alpha)\right]
	\end{align*}
	where $X_{\xi} = \{\alpha \in \rset^n: \norm{\sum_i \alpha_i y_i x_{ji}}_2 \le \gamma d_j\}$ for $\xi \in [m] $ and  $X_{m+1} = \{\alpha \in \rset^n: \alpha^Ty = 0\}, X_{m+2} = \{\alpha \in \rset^n: 0 \le \alpha \le C\}$.
	In [2] can be found the analysis of the above problem, without the regularization term $\frac{\lambda}{2}\norm{\alpha}_2^2$. However, small $\ell_2$ regularizations are often used to improve problem conditioning (or boost algorithms performance) under slight perturbations in the solution.
	\item[$(III)$] Network lasso problem [11, Section 2]: let the graph $(V,E)$, where $V$ are vertices and $E$ the edges, then the network lasso problem minimizes a cost given by the sum of a vertices component (established by application) and an edge component which penalize the difference the variables at adjacent nodes:
	\begin{align*}
	\min\limits_x \mathbb{E}_{\xi \in \mathcal{V}} [f_{\xi}(x_{\xi})] + \lambda \mathbb{E}_{(\xi,\nu) \in E} \; w_{\xi \nu}\norm{x_{\xi}-x_{\nu}}_2, 
	\end{align*}
	where $\xi,\nu$ are random variables with values in $\mathcal{V}$.
	For example, in Event Detection in time series the vertex cost is strongly convex: $f_{\xi}(x_{\xi}) = \norm{x_{\xi} - \bar{x}_{\xi}}^2_2 + \mu \norm{x_{\xi}}_2$ (see [11, Section 5.3]).
\end{itemize}

\noindent Further we properly analyze several well-known classes of functions and prove that they satisfy the Assumption \ref{assump_localgrowth}.

\subsection{Classes of functions satisfying weak linear regularity}\label{sec_funcclass}

\noindent  Further we will enumerate some classes of functions which are often found in the optimization and machine learning literature, and then show that functions from each class satisfy the weak linear regularity. We will center our examples on the constrained model: 
\begin{align*}
\min\limits_{x \in \rset^n}  \;  \mathbb{E}_{\xi \in \Omega_1}[f(x;\xi)] \qquad \text{s.t.} \; \bigcap_{\xi \in \Omega_2} X_{\xi},
\end{align*}


\subsubsection{Quadratically growing functions with linearly regular constraints}

A well known strong convexity relaxation which is often used to show the linear convergence for the deterministic first-order methods is the quadratic growth property. In this section we prove that usual quadratic growth together with a smoothness assumption implies the weak linear regularity property. 

\begin{assumption}\label{assump_qgoriginal}
	The function $f$ satisfies quadratic growth with constant $\sigma_f$ if the following relation holds:
	\begin{align*}
	f(x) - f^* \ge \frac{\sigma_f}{2}\text{dist}_{X^*}^2(x) \qquad \forall x \in X.
	\end{align*}
\end{assumption}
\noindent From convexity of function $f$ we have
$\langle \nabla f(x), x - \pi_{X^*}(x) \rangle \ge f(x) - f^* \ge \frac{\sigma_f}{2}\text{dist}_{X^*}^2(x),$
which by Cauchy-Schwartz inequality implies:
\begin{align}\label{grad_qg}
\norm{\nabla f(x)} \ge \frac{\sigma_f}{2}\text{dist}_{X^*}(x) \qquad \forall x \in X.
\end{align}
\begin{assumption}\label{ass_lipgrad}
	Each function $f(\cdot;\xi)$ has Lipschitz continuous gradients with constant $L_{\xi}$, i.e. there exists $L_{\xi}$ such that the following relation holds:
	\begin{align*}
	\norm{\nabla f(x;\xi) - \nabla f(y;\xi)} \le L_\xi \norm{x-y}, \quad \forall x,y \in \rset^n.
	\end{align*}	
\end{assumption}
\noindent It is easy to see that under Lipschitz continuity, there exist a relation between norms of gradients $\nabla f$ and $\nabla f_{\mu}$.
\begin{theorem}\label{lemma_gradorder}
	Let Assumption \ref{ass_lipgrad}, then the following relation holds:
	\begin{align*}
	\mathbb{E}\left[ \frac{\norm{\nabla f(x;\xi)}^2}{(1 + L_{\xi} \mu)^2} \right] \le \mathbb{E}\left[\norm{\nabla f_{\mu}(x;\xi)}^2\right] \qquad \forall x \in \rset^n.
	\end{align*} 	
\end{theorem}
\begin{proof}
	Based on the Lipschitz gradient assumption and the triangle inequality we have:
	\begin{align*}
	\norm{\nabla f(x;\xi)} 
	& \le \norm{\nabla f(x;\xi) - \nabla f(z_{\mu}(x;\xi);\xi)} + \norm{\nabla f(z_{\mu}(x;\xi);\xi)} \\
	& \overset{L.c.g.}{\le} L_{\xi}\norm{x - z_{\mu}(x;\xi)} + \norm{\nabla f(z_{\mu}(x;\xi);\xi)} \\
	& = (1 + L_{\xi}\mu) \norm{\nabla f_{\mu}(x;\xi)}.
	\end{align*}
	By taking expectation in both sides we obtain the result.
\end{proof}
\noindent Further we derive the weak linear regularity relation for the extended function $F$.
\begin{theorem}\label{lemma_qg_proxqg}
	Let Assumptions \ref{assump_qgoriginal}-\ref{ass_lipgrad} hold. Also assume that the constraints $\{X_{\xi}\}_{\xi \in \Omega}$ are linearly regular with constant $\sigma_X$. Then, the composite objective $F$ satisfies weak linear regularity (Assumption \ref{assump_localgrowth})  with constants:
	\begin{align*}
	\sigma_{F,\mu} & = \frac{\sigma_f \mu \sigma_X}{\sigma_f \mu^2 + 8(1 + 2\sigma_X)(1 + \mu L_{\max})^2 } \\
	\beta & = \mathbb{E}\left[\norm{\nabla f(x^*;\xi)}^2 \right] + \left[ \frac{1}{\sigma_X} + \frac{1}{4\sigma_X^3}\left(1 + \frac{\sigma_f}{4L_{\max}^2} \right)^2 \right]\norm{\nabla f(x^*)}^2 , 
	\end{align*}
	where $L_{\max} = \sup_{\xi \in \Omega}  L_{\xi}$.
\end{theorem}

\begin{proof}
The proof can be found in the Appendix.
\end{proof}

\begin{remark}\label{rem_qg}
	In the unconstrained finite case, when $X = \rset^n, \Omega = [m]$, a well-known class of problems having quadratic growth is given by:
	$$ F(x) := \mathbb{E}[g(A_{\xi}x;\xi)],$$
	where $g(\cdot;\xi)$ is strongly convex with constant $\sigma_{\xi} > 0$ and has Lipschitz gradients. However, for completeness we sketch the arguments for proving the property.	From the strong convexity property of $g(\cdot;\xi)$ we derive the restricted-strong convexity:
	\begin{align}\label{gax_qg}
	g(A_{\xi}x;\xi) \ge g(A_{\xi}y;\xi) + \langle A_{\xi}^T \nabla g(A_{\xi}y;\xi), x - y \rangle + \frac{\sigma_{\xi}}{2}\norm{A_{\xi}(x-y)}^2, \quad \forall x, y \in \rset^n.
	\end{align}
	Let $x_1^*, x_2^* \in X^*$, then setting $x = x_1^*$ and $y = x_2^*$ and by taking expectation in both sides, results:
	\begin{align*}
	F(x_1^*) 
	&\ge F(x_2^*) + \langle \nabla F(x_2^*), x_1^* - x_2^* \rangle + \frac{1}{2}\langle x_1^* - x_2^*, \mathbb{E}[\sigma_{\xi}A_{\xi}^TA_{\xi}](x_1^* - x_2^*) \rangle \\
	&\ge F(x_2^*) + \frac{1}{2}\langle x_1^* - x_2^*, \mathbb{E}[\sigma_{\xi}A_{\xi}^TA_{\xi}](x_1^* - x_2^*) \rangle.
	\end{align*}
	Since $F(x_1^*) = F(x_2^*)$, the relation yields that there are unique $y^*_{\xi} = A_{\xi}x^*$ for all $x^* \in X^*, \xi \in \Omega$. Therefore, we clearly have that $X^* = \{x: A_{\xi} x = y^*_{\xi} \}$ and we can use the Hoffman's bound to derive: there is $\kappa_c>0$ such that $\mathbb{E}[\norm{A_{\xi}x - y^*_{\xi}}^2] \ge\kappa_c \text{dist}_{X^*}^2(x)$. Using this fact in \eqref{gax_qg} with $y = x^*$  then we reach our conclusion.
	
	\noindent For example, linear regression can be casted by the above particular model, i.e.
	$$ F(x) = \mathbb{E}[(a_{\xi}^Tx - b_{\xi})^2].$$
	Notice that we do not assume the interpolation property and thus the model admits systems $Ax = b$ without solution. Based on standard arguments from literature, it can be shown that the above unconstrained model can be further generalized to linear constraints and polyhedral regularization (see \cite{DruLew:18,XuLin:17, NecNes:16}). 
	Many practical applications can be casted into one these models (see \cite{DruLew:18,XuLin:17}), such as LASSO-regularized regression: $\min_x \norm{Ax-b}^2 +\lambda \norm{x}_1$ \cite{fista}, support vector machine with polyhedral regularization: $\min_x \frac{1}{m} \sum\limits_{i=1}^m \max\{0, a_{\xi}x - b_{\xi}\} + \lambda \norm{x}_1$ \cite{XuLin:17, ShaSin:11}, constrained linear regression $\min_x \; \norm{Ax-b}^2 \; \text{s.t.} \; Cx \le d$ \cite{CenChe:12}, etc.
\end{remark}


\subsubsection{Restricted strongly convex function with general constraints}

A slightly more restrictive class of functions is described by the restricted strong convexity (RSC) property, which has been extensively analyzed in \cite{RyuBoy:16,TouTra:16,YanLin:18,AsiDuc:18}. 
Notice that in \cite{AsiDuc:18} the authors derive complexity of SPP under RSC for each $f(\cdot;\xi)$ and strong convexity on $F$, which might indirectly allow indicator functions. However, our completely different analysis allows non-strongly convex objectives and a direct particularization to CFPs setting, i.e. indicator functions.

\begin{definition} The function $f(\cdot;\xi)$ is $M_{\xi}-$restricted strongly convex if there exists $M_{\xi} \succeq 0$ such that
	\begin{align*}
	f(x;\xi) \ge f(y;\xi) + \langle g_f(y;\xi),x - y\rangle + \frac{1}{2}\langle x - y, M_{\xi}(x-y) \rangle, \qquad \forall x,y \in \rset^n.
	\end{align*}
\end{definition}
\noindent Although it describes the behavior of each component $f(\cdot;\xi)$, this restricted convexity allows the elimination of other smoothness assumptions, which for the previous quadratic growth is not the case. Next, we show that the smoothing function also inherits the restricted strong convexity with specific parameter.
\begin{lemma}\label{lemma_restrictedscqg}
	Let $f(\cdot;\xi)$ be $M_{\xi}-$restricted strongly convex (RSC). 
	Then, given $\mu > 0$, the approximation $f_{\mu}$ is $\mathbb{E}\left[\frac{M_{\xi}}{\lambda_{\max}(I + \mu M_{\xi})} \right]-$RSC: 
	\begin{align*}
	f_{\mu}(x) \ge f_{\mu}(y) + \langle \nabla f_{\mu}(y), x - y\rangle + \frac{1}{2} \langle x - y,\mathbb{E}\left[\frac{M_{\xi}}{\lambda_{\max}(I + \mu M_{\xi})} \right] (x - y)\rangle \qquad \forall x,y \in \rset^n.
	\end{align*}
\end{lemma}
\begin{proof}
	The proof can be found in the appendix.
\end{proof}
\noindent Recall the fact that for a strongly convex function with constant $\sigma$, its Moreau smoothing remains strongly convex with constant $\frac{\sigma}{1 + \mu\sigma}$, see \cite{Roc:70}. Thus it is obvious that the RSC matrix $\frac{M_{\xi}}{\lambda_{\max}(I + \mu M_{\xi})}$ is a natural generalization of the strong convexity constant.
\noindent The $M_{\xi}$-RSC property do not require that the functional component $f(\cdot;\xi)$ to be strongly convex since $M_{\xi} \succeq 0$. However, if $M_f = \mathbb{E}[M_{\xi}] \succ 0$, then $f$ is $\lambda_{\min}(M_f)-$strongly convex.  
Although we are more interested in the non-strongly convex objective functions, further we additionally analyze for completeness the strongly convex case. Moreover, we show that the linear regularity brings significant advantages when $F$ is strongly convex.

\noindent Now we provide the result stating that, under RSC, the extended objective function $F$ satisfies the weak linear regularity property.

\begin{theorem}\label{th_rsc_proxqg}
	Let $f(\cdot;\xi)$ be $M_{\xi}-$RSC and denote $\hat{M}_f = \mathbb{E}\left[ \frac{M_{\xi}}{\lambda_{\max}(I_n + \mu M_{\xi})}\right]$. Then the composite function $F (x) : = \mathbb{E} \left[f(x;\xi)\right] + \mathbb{E} \left[\mathbb{I}_{X_{\xi}}(x)\right] $ satisfies the following properties:
	
	\noindent $(i)$ Assume that $M_f \succeq 0$, $y^* = M_f^{1/2}x^*$ is unique for all optimal points $x^*$ and $\{X_{\xi}\}_{\xi \in \Omega}$ are linearly regular with constant $\sigma_X$. By denoting $\hat{X} = \{x: M_f^{1/2}x = y^*\}$, assume that the optimal set is defined by $X^* = X \cap \hat{X}$ and $(\hat{X},\{X_{\xi}\}_{\xi \in \Omega})$ are linearly regular with constant $\sigma_{f,X}$. Then $F$ satisfies Assumption \ref{assump_localgrowth} with constants:
	\begin{align*}
	\sigma_{F,\mu} \!=\! \frac{\min\{1, \lambda_{\max}(M_{\max})\}}{2(1 + \mu \lambda_{\max}(M_{\max}))} \sigma_{f,X}, \;\;   
	\beta = \frac{\mu}{2} \left(1 +\! \frac{2}{\sigma_X}\! \right)\mathbb{E}\left[\norm{g_f(x^*;\xi)}^2 \right] \! +\! \frac{\mu}{2\sigma_X}\norm{g_f(x^*)}^2,
	\end{align*}
	where $M_{\max} = \sup\limits_{\xi \in \Omega} \lambda_{\max}(M_{\xi})$.
	
	\noindent $(ii)$ If $M_f \succ 0$, then $F$ satisfies Assumption \ref{assump_localgrowth} with constants:
	$$\sigma_{F,\mu} = \lambda_{\min}(\hat{M}_F) \qquad \beta = \frac{1}{2}\mathbb{E}[\norm{g_F(x^*;\xi)}^2].$$
	
	\noindent $(iii)$ If $M_f \succ 0$ and $\{X_{\xi}\}_{\xi \in \Omega}$ are linearly regular with constant $\sigma_X$, then $F$ satisfies weak linear regularity assumption \ref{assump_localgrowth} with constants:
	$$\sigma_{F,\mu} = \lambda_{\min}(\hat{M}_F) \qquad \beta = \frac{1}{2}\mathbb{E}[\norm{g_f(x^*;\xi)}^2] + \frac{1}{2\sigma_X}\norm{\mathbb{E}[g_f(x^*;\xi)]}^2.$$
\end{theorem} 

\begin{proof}
The proof can be found in Appendix.
\end{proof}

\begin{remark}
	It is clear that the assumptions of Theorem \ref{th_rsc_proxqg} $(i)$ hold for similar models as in the previous quadratic growth case:
	\begin{align*}
	\min\limits_{x \in \rset^n} & \; \mathbb{E}[g(A_{\xi}x;\xi)] \\
	\text{s.t.} & \;\; C_{\xi}x \le d_{\xi}, \quad \forall \xi \in \Omega.
	\end{align*}
	where $g(\cdot;\xi)$ is strongly convex with constant $\sigma_{\xi} > 0$ and $\Omega = \{1, \cdots, m\}$. However, the Lipschitz gradient assumption is not necessary any more. 
	Since $M_{\xi} = A_{\xi}^TA_{\xi}$, using similar arguments as in Remark \ref{rem_qg}, we have that  there are unique $y^*_{\xi} = A_{\xi}x^*$ and $y^* = M_f^{1/2} x^*$ for all $x^* \in X^*, \xi \in \Omega$. By observing that 
	$$M_f^{1/2} x = y^* \;\; \Leftrightarrow \;\; \langle x - x^*, \mathbb{E}[A_{\xi}^TA_{\xi}](x - x^*) \rangle = 0 \;\; \Leftrightarrow \;\; A_{\xi}x = y^*_{\xi}  $$
	we clearly have that $X^* = \{x^*: M_f^{1/2} x^* = y^*, C_{\xi}x \le d_{\xi}, \xi \in \Omega \}$.
\end{remark}
\noindent The same above arguments extend further to the case when $g(\cdot;\xi)$ is strongly convex on any compact set \cite{YanLin:18}. In this way the model covers the sparse robust regression problem:
\begin{align*}
 \min\limits_{w \in \rset^n} \; \mathbb{E}_{\xi}[\;|a_{\xi}w - b_{\xi} |^p\;] + \lambda \norm{w}_1, \qquad p \in (1,2).
\end{align*}
	


	\noindent It is important to observe that, under restricted strong convexity with $M_f \succ 0$, the weak linear regularity hold for general convex constraints, i.e. even if the constraints are not linearly regular.  In \cite{PatNec:18}, the linear regularity property of the constraint sets was essential to get the sublinear convergence rates. Also in  \cite{YanLin:18} no regularity is required, but a full projection on the entire feasible set is performed  at each iteration.


\subsubsection{Convex feasibility problems}

\noindent The most intuitive function class in our framework proves to be the indicator functions class, where $f(x;\xi) = 0$. Let $\{X_{\xi}\}_{1 \le \xi \le m}$ be a finite collection of convex sets and $X = \Large\cap_{\xi} X_{\xi} \neq \emptyset$. Under these terms yields that:
\begin{align*}
F_{\mu}(x) = \mathbb{E}\left[\text{dist}^2_{X_{\xi}}(x)\right], \qquad F_{\mu}^* = F^* = 0, \qquad  X_{\mu}^* = X^* = X.
\end{align*}
Then, it is easy to see that under the linear regularity assumption, the weak linear regularity property is immediately implied
\begin{align}\label{cfp_qg}
F_{\mu}(x) - F_{\mu}^* =  \mathbb{E}\left[\frac{1}{2\mu}\text{dist}^2_{X_{\xi}}(x)\right] \ge \frac{\sigma_X}{2\mu} \text{dist}_{X}^2(x) \qquad \forall x \in \rset^n,
\end{align} 
with corresponding constants $\sigma_{F,\mu} = \frac{\sigma_X}{\mu}$ and $\beta = 0$. Most common example of linearly regular sets are the polyhedral sets. Also in the case when the intersection has nonempty interior and contains a ball of radius $\delta$, then the linear regularity holds with constant $\sigma_X$ dependent on $\delta$, see \cite{Ned:10}.



\section{Stochastic Proximal Point algorithm}

In the following section we present a one-step stochastic proximal point scheme for problem \eqref{problem} and analyze its convergence behavior. Observing that $\pi_X(\cdot)$ is in fact a particular proximal operator associated to $\mathbb{I}_X(\cdot)$, then
the stochastic alternating projection iteration
$x^{k+1}  = \pi_{X_{\xi_k}}(x^k)$ is a particular expression of the more general SPP iteration:
\begin{align*}
x^{k+1} 
& = z_{\mu_k}(x^k;\xi_k).
\end{align*}
Thus SPP algorithm might be interpreted as a natural generalization of SAP (see e.g. \cite{CenChe:12}). 
Since in general $X_{\mu}^* \neq X^*$, the smoothing parameter should be decreased in order to guarantee convergence towards the minimizer of the original problem. 

\noindent Let initial iterate $x^0 \in \rset^n$ and $\{\mu_k\}_{k \ge 0}$ be a nonincreasing positive sequence of
stepsizes.

\begin{flushleft}
	\textbf{ Stochastic Proximal Point (SPP) ($x_0, \{\mu_k\}_{k\ge 0}$) }: \quad For $k\geq 1$ compute\\
	1. Choose randomly $\xi_k \in \Omega$ w.r.t. probability distribution $\mathbb{P}$\\
	2. Update: $x^{k+1}  =  z_{\mu_k}(x^{k};\xi_k)$.
\end{flushleft}
\noindent Note that there are many practical cases when the proximal operator $z_{\mu}$ can be computed easily or even has a closed form. To exemplify a few: 
\begin{itemize}
	\item[$(i)$]  the least-square loss $F(x;\xi) = \frac{1}{2}(a^T_{\xi}x - b_{\xi})^2$, $z_{\mu}(x;\xi) = x - \frac{\mu (a^T_{\xi} - b_{\xi})}{1 + \mu \norm{a_{\xi}}^2} a_{\xi}$; 
	\item[$(ii)$] regularized hinge-loss $F(x;\xi) = \max\{0, a^T_{\xi}x - b_{\xi} \} + \frac{\lambda}{2}\norm{x}^2_2$, $z_{\mu}(x;\xi) = \frac{1}{1+\lambda\mu}\left(x - \mu a_{\xi} s \right)$, where $s = \pi_{[0,1]}\left( \frac{1-\lambda \mu}{\mu(1 + \lambda \mu)}\frac{a^T_{\xi}x}{\norm{a_{\xi}}^2} - \frac{1 + \lambda\mu}{\mu}\frac{b}{\norm{a_{\xi}}^2} \right)$.
	\item[$(iii)$] halfspace: $ H_{\xi} = \{x: a_{\xi}^Tx \le b_{\xi}\}, F(x;\xi) = \mathbb{I}_{H_{\xi}}(x), z_{\mu}(x;\xi) = x - \frac{\max\{0,a_{\xi}^Tx - b_{\xi}\}}{ \norm{a_{\xi}}^2 } a_{\xi}$.
\end{itemize} 


\section{Iteration complexity of SPP under weak linear regularity}\label{sec_mainresults}

\noindent In this section  we derive sublinear and linear convergence rates of stochastic proximal point scheme under various convexity and regularity conditions of the objective function. 

\noindent Further we present some lower bounds on the gap between the smooth approximations and the optimal values of the objective function.
\begin{lemma}\label{lemma_auxres}
	Given $\mu > 0$, let $\{X_{\xi}\}_{\xi \in \Omega}$ be some closed convex sets and  recall $F(x) = f(x) + \mathbb{E}[\mathbb{I}_{X_{\xi}}(x)]$. 
	Then, under Assumption \ref{assump_basic}, the following relations hold:
	\begin{enumerate}
		\item[$(i)$] $F_{\mu}(x) \le F(x) \quad \forall x \in \rset^n$,
		\item[$(ii)$] $ F^* - F_{\mu}(x) \le \frac{\mu}{2}\mathbb{E}\left[\norm{g_F(x^*;\xi)}^2 \right] \le \frac{\mu}{2}\mathcal{S}^*_F \quad \forall x \in \rset^n.$ 
		\item[$(iii)$] Let linear regularity hold for $\{X_{\xi}\}_{\xi \in \Omega}$ with constant $\sigma_X > 0$, then \\
		$ F^* - F_{\mu}(x) \le  \mathbb{E}\left[\frac{\mu}{2}  \norm{g_f(x^*;\xi)}^2 \right] +  \frac{\mu}{2\sigma_X}\norm{g_f(x^*)}^2 \quad \forall x \in \rset^n.$ 	
	\end{enumerate}
\end{lemma}
\begin{proof}
The proof can be found in Appendix.
\end{proof}


\noindent In order to establish the SPP convergence we provide the following relation.

\begin{theorem}\label{th_recurence}
Let Assumptions \ref{assump_basic} and \ref{assump_localgrowth} hold. The sequence $\{x^k\}_{k \ge 0}$ generated by SPP satisfies:
	\begin{align*}
\mathbb{E}\left[\text{dist}_{X^*}^2(x^{k+1})\right] 
 \! \le \!  \left[\prod_{i=0}^k (1 \!-\! \mu_i \sigma_{F,\mu_0})\right]\! \text{dist}_{X^*}^2(x^0) \!+\! (\mathcal{S}^*_F \!+\! 2\beta) \sum\limits_{i=0}^k \left[ \prod\limits_{j=i+1}^{k} (1 \!-\! \mu_j \sigma_{F,\mu_0})\right] \mu_i^2.
\end{align*}
\end{theorem}
\begin{proof}
Note that $F(\cdot;\xi) + \frac{1}{2\mu}\norm{\cdot - x}^2$ is strongly convex with constant $\frac{1}{\mu}$ therefore:
\begin{align*}
 F(\pi_{X^*} & (x^k)  ;\xi_k)  + \frac{1}{2\mu_k}\norm{\pi_{X^*}(x^k) - x^k}^2 \\
&  \ge F(z_{\mu_k}(x^k;\xi_k);\xi_k) + \frac{1}{2\mu_k}\norm{z_{\mu_k}(x^k;\xi_k) - x^k}^2 + \frac{1}{2\mu_k}\norm{z_{\mu_k}(x^k;\xi_k) - \pi_{X^*}(x^k)}^2 \\
& = F_{\mu_k}(x^k;\xi_k) + \frac{1}{2\mu_k} \norm{z_{\mu_k}(x^k;\xi_k) - \pi_{X^*}(x^k)}^2.
\end{align*}	
By taking expectation over $\xi_k$ in both sides, then for any $x^*_{\mu_k} \in X^*_{\mu_k}$ we obtain:
\begin{align}
\mathbb{E}_{\xi_k} \left[\norm{x^{k+1} - \pi_{X^*}(x^k)}^2\right] 
& \le \norm{x^k - \pi_{X^*}(x^k)}^2 + 2\mu_k \left(F(\pi_{X^*}(x^k)) - F_{\mu_k}(x^k) \right) \nonumber \\ 
& = \text{dist}_{X^*}^2(x^k) + 2\mu_k \left(F^* - F_{\mu_k}(x^*_{\mu_k}) \right) + 2\mu_k \left(F^*_{\mu_k} - F_{\mu_k}(x^k) \right) \nonumber \\
& \hspace{-45pt} \overset{Lemma \; \ref{lemma_auxres}\; (ii)}\le \text{dist}_{X^*}^2(x^k) +  \mu_k^2 \mathbb{E}_{\xi} \left[ \norm{g_F(x^*;\xi)}^2\right] + 2\mu_k \left(F^*_{\mu_k} - F_{\mu_k}(x^k) \right) \nonumber \\
& \hspace{-50pt}  \overset{Assumption \;  \ref{assump_localgrowth}}{\le} (1 - \mu_k \sigma_{F,\mu_k})\text{dist}_{X^*}^2(x^k)  + \mu_k^2\left(\mathbb{E}_{\xi}\left[ \norm{g_f(x^*;\xi)}^2\right] +  2 \beta \right). \label{th_recurrence_rel1}
\end{align}
By observing that $\norm{x^{k+1} - \pi_{X^*}(x^k)} \ge \text{dist}_{X^*}(x^{k+1})$ then by taking full expectation in both sides of \eqref{th_recurrence_rel1} we get:
\begin{align}
\mathbb{E}\left[\text{dist}_{X^*}^2(x^{k+1})\right] 
& \le  (1 - \mu_k \sigma_{F,\mu_k}) \mathbb{E}\left[\text{dist}_{X^*}^2(x^k)\right] + \mu_k^2\left( \mathbb{E}\left[\norm{g_F(x^*;\xi)}^2\right] + 2 \beta\right)\nonumber \\
& \le  (1 - \mu_k \sigma_{F,\mu_0}) \mathbb{E}\left[\text{dist}_{X^*}^2(x^k)\right] + \mu_k^2 \left( \mathbb{E}\left[\norm{g_F(x^*;\xi)}^2\right] + 2 \beta\right), \nonumber 
	\end{align}
where we used that $\sigma_{F,\mu_k} \ge \sigma_{F,\mu_0}$, since $\sigma_{F,\mu} $ is nonincreasing in $\mu$.
Now for simplicity if we denote $\theta_k = (1 - \mu_k\sigma_{F,\mu_0})$, then we can further derive:
	\begin{align*}
	\mathbb{E}\left[\text{dist}_{X^*}^2(x^{k+1})\right] 
	& \le  \theta_k \mathbb{E}\left[\text{dist}_{X^*}^2(x^{k})\right] + \mu_k^2 (\mathcal{S}^*_F + 2\beta) \\
	& \le  \left(\prod_{i=0}^k \theta_i\right) \text{dist}_{X^*}^2(x^0) + (\mathcal{S}^*_F + 2\beta) \sum\limits_{i=0}^k \left(\prod\limits_{j=i+1}^{k} \theta_j\right) \mu_i^2,
	\end{align*}
which confirms our result.	
\end{proof} 
\begin{remark}
The universal upper bound provided in Theorem \ref{th_recurence} will be used to generate sublinear rate for non-interpolation context, i.e. $\mathcal{S}_F^* + \beta > 0$, and linear convergence rates for constant stepsize SPP under interpolation assumption, i.e. $\mathcal{S}_F^* = \beta = 0$. 
\end{remark}


\subsection{Sublinear convergence rate}

\noindent The sublinear convergence rate for SPP under the weak linear regularity property can be easily obtained from Theorem \ref{th_recurence}. 

\begin{theorem}\label{th_sublinearrate}
Let Assumptions \ref{assump_basic}, \ref{assump_localgrowth} hold. Also let the decreasing stepsize sequence $\mu_k = \frac{\mu_0}{k^{\gamma}}$ and $\{x^k\}_{k \ge 0}$ be the sequence generated by SPP($x^0,\{\mu_k\}_{k \ge 0}$). Then, for any $k > 0$, the following relation holds:
\begin{align*}
(i) \; \text{If} \; \gamma \in (0,1): \qquad 
\mathbb{E}[\norm{ x^{k} - x^*}^2]
\le \mathcal{O}\left(\frac{1}{k^{\gamma}}\right)
\end{align*}
\begin{align*}
(ii) \; \text{If} \; \gamma = 1: \qquad 
\mathbb{E}[\norm{ x^{k} - x^*}^2] \le
\begin{cases}
\mathcal{O}\left(\frac{1}{ k}\right)     & \text{if} \; \mu_0\sigma_{F,\mu_0}>e-1 \\
\mathcal{O}\left(\frac{\ln{k}}{ k}\right)     & \text{if} \; \mu_0\sigma_{F,\mu_0}=e-1\\
\mathcal{O}\left(\frac{1}{k}\right)^{2 \ln
	(1 + \mu_0\sigma_{F,\mu_0})}  & \text{if} \; \mu_0\sigma_{F,\mu_0} < e-1.
\end{cases}
\end{align*}
\end{theorem}
\begin{proof}
For simplicity denote $\theta_k = (1 - \mu_k\sigma_{F,\mu_0})$ , then Theorem \ref{th_recurence} implies that:
	\begin{align*}
	\mathbb{E}\left[\text{dist}_{X^*}^2(x^{k+1})\right] 
	& \le  \left(\prod_{i=0}^k \theta_i\right) \text{dist}_{X^*}^2(x^0) + \left(\mathcal{S}^*_F+\beta \right) \sum\limits_{i=0}^k \left(\prod\limits_{j=i+1}^{k} \theta_j\right) \mu_i^2.
	\end{align*}
	By using the Bernoulli inequality $ 1- tx \le \frac{1}{1 + tx} \le (1 + x)^{-t}$ for $t \in [0,1], x \ge 0$, then we have:
	\begin{align}\label{bern_conseq1}
	\prod\limits_{i=l}^u \theta_i 
	& = \prod\limits_{i=l}^u \left(1 - \frac{\mu_0}{i^{\gamma}} \sigma_{F,\mu_0}\right)  \le \prod\limits_{i=l}^u (1 + \mu_0 \sigma_{F,\mu_0})^{-1/i^{\gamma}} 
	=  (1 + \mu_0 \sigma_{F,\mu_0})^{- \sum\limits_{i=l}^u \frac{1}{i^{\gamma}}}.
	\end{align}
	On the other hand, if we use the lower bound
$	\sum\limits_{i=l}^u \frac{1}{i^{\gamma}} \ge \int\limits_{l}^{u + 1} \frac{1}{\tau^{\gamma}} d\tau =  \varphi_{1-\gamma}(u+1) - \varphi_{1-\gamma}(l)$ 	then we derive:
	\begin{align}
	& \sum\limits_{i=0}^k \left(\prod\limits_{j=i+1}^{k} \theta_j\right) \mu_i^2  = \sum\limits_{i=0}^m \left(\prod\limits_{j=i+1}^{k} \theta_j\right) \mu_i^2 + \sum\limits_{i=m+1}^k \left(\prod\limits_{j=i+1}^{k} \theta_j\right) \mu_i^2 \nonumber \\
	& \overset{\eqref{bern_conseq1}}{\le} \sum\limits_{i=0}^m (1 + \mu_0 \sigma_{F,\mu_0})^{  \varphi_{1-\gamma}(i+1) - \varphi_{1-\gamma}(k)  } \mu_i^2 + \mu_{m+1} \sum\limits_{i=m+1}^k \left[\prod\limits_{j=i+1}^{k} (1 - \mu_j\sigma_{F,\mu_0}) \right] \mu_i \label{recc_decompose} \end{align}
	Furthermore, by using that:
	\begin{align}\label{recc_aux1}
	\sum\limits_{i=0}^m  \mu_i^2 \le \mu_0^2 \sum\limits_{i=0}^m  \frac{1}{i^{2\gamma}} \le \frac{m^{1- 2\gamma} - 1}{1 - 2\gamma} = \varphi_{1 - 2\gamma}(m) 
	\end{align}
	and that 
	\begin{align}\label{recc_aux2}
 & \sum\limits_{i=m+1}^k \left[\prod\limits_{j=i+1}^{k} (1 - \mu_j\sigma_{F,\mu_0}) \right] (1 - (1- \sigma_{F,\mu_0}\mu_i)) \nonumber \\
 & = \sum\limits_{i=m+1}^k \left[\prod\limits_{j=i+1}^{k} (1 - \mu_j\sigma_{F,\mu_0})  - \prod\limits_{j=i}^{k} (1 - \mu_j\sigma_{F,\mu_0}) \right]  =  1  - \prod\limits_{j=m+1}^{k} (1 - \mu_j\sigma_{F,\mu_0})   \le  \frac{\mu_{m+1}}{\sigma_{F,\mu_0}}.
	\end{align}
	By denoting the second constant $\tilde{\theta}_0 = \frac{1}{1+\mu_0 \sigma_{F,\mu_0}}$, then \eqref{recc_decompose}-\eqref{recc_aux2} implies:
	\begin{align*}
	& \mathbb{E}\left[\text{dist}_{X^*}^2(x^{k+1})\right] 
	\le \tilde{\theta}_0^{\varphi_{1-\gamma}(k)} \text{dist}_{X^*}^2(x^0) \\
	& \hspace{2cm} + \tilde{\theta}_0^{  \varphi_{1-\gamma}(k) - \varphi_{1-\gamma}(m)  } \varphi_{1 - 2\gamma}(m)(\mathcal{S}^*_F+\beta) + \frac{\mu_{m+1}}{\sigma_{F,\mu_0}} (\mathcal{S}^*_F+\beta).
	\end{align*}
	To derive the explicit case-wise convergence rate order we analyze upper bounds on function $\phi$ and follow similar steps as in \cite[Corrolary 15]{PatNec:18}. 

\end{proof}

\noindent A similar convergence rate result can be found \cite{PatNec:18} under the Lipschitz gradient and strong convexity assumptions. However, our analysis is much simpler and requires only weak linear regularity, which holds even for some particular nonsmooth non-strongly convex objective functions. 

\subsection{Linear convergence rate}

We show that the sublinear rate can be further improved under additional stronger assumptions related to the interpolation setting.
\begin{assumption}\label{assumpt_shared_min}
The functional components $F(\cdot;\xi)$ share common minimizers, i.e. for any $x^* \in X^*$
\begin{align*}
   0 \in \partial F(x^*;\xi) \qquad \forall \xi \in \Omega.
\end{align*}
\end{assumption}
\noindent The interpolation condition is typical for CFPs, where is aimed to find a common point of a collection of convex sets, i.e. $f(\cdot;\xi) = \mathbb{I}_{X_{\xi}}(\cdot)$ and $X^* = \bigcap_{\xi \in \Omega} X_{\xi}$. 
Examples satisfying Assumption \ref{assumpt_shared_min} are:
\begin{align*}
(i) & \; F(x;\xi) = \mathbb{I}_{H_{\xi}}(x), \qquad H_{\xi} = \{x: a_{\xi}x \le b_{\xi}\},\\
(ii) & \; F(x;\xi) = (a_{\xi}x - b_{\xi})^2 
\end{align*}
assuming that there exists $z$ such that: for $(i)$ $a_{\xi}z \le b_{\xi}, \forall \xi$, for $(ii)$ $a_{\xi}z = b_{\xi}, \forall \xi$.

\noindent In \cite{MaBas:18}, the linear convergence of SGD has been extensively analyzed for the interpolation least-squares problems. An immediate consequence of Assumption \ref{assumpt_shared_min} is that given any optimal $x^*$ we can find  subgradients $g_F(x^*;\xi)$ for each $\xi$ such that 
$\mathbb{E} \left[\norm{g_F (x^*;\xi)}^2 \right]  =  0, \forall x^* \in X^*. $
Further by taking into account that the Moreau envelope preserves the set of minimizers corresponding to each functional component, then we have 
\begin{align*}
X^* = X^*_{\mu} \quad \text{and} \quad \mathbb{E} \left[\norm{\nabla f_{\mu}(x^*;\xi)}^2 \right] = 0 \qquad \forall x^* \in X^*, \mu>0. 
\end{align*}
This fact implies that the decaying stepsize of the SPP iteration is not necessary any more. A straightforward application of Theorem \ref{th_recurence} leads to the following constant decrease:
\begin{corrolary}\label{corr_linconv}
Let Assumption \ref{assump_localgrowth} hold with $\beta = 0$ . If also the Assumption \ref{assumpt_shared_min} holds, then Theorem \ref{th_recurence} implies that the sequence $\{x^k\}_{k \ge 0}$ generated by constant stepsize SPP satisfies:
\begin{align*}
\mathbb{E} \left[\text{dist}_{X^*}^2(x^{k})\right] 
\le (1 - \mu\sigma_{F,\mu})^k \mathbb{E} \left[\text{dist}_{X^*}^2(x^0)\right].
\end{align*}
\end{corrolary}
\noindent As proved in the previous sections, the indicator functions w.r.t. linearly regular sets, the restricted strongly convex functions and some particularly structured quadratically growing functions satisfy the Assumption \ref{assump_localgrowth} with $\beta = \mathcal{O}\left( \mathbb{E}[\norm{g_F(x^*;\xi)}^2]\right)$, which is possibly vanishing in the interpolation context when the Assumption \ref{assumpt_shared_min} holds. 
It seems that using other analysis from \cite{PatNec:18,TouTra:16} cannot be guaranteed that SPP converges linearly in the interpolation settings. The work \cite{AsiDuc:18} is centered on behavior of vanishing stepsize SPP for interpolation problems and obtain impressive complexity and stability results. However, we focus on quadratic growth relaxations of strong convexity assumption which allow the a unified treatment of interpolation and non-interpolation contexts.

\noindent Let us consider CFP case $f = 0$. In this case the SPP iteration becomes the vanilla SAP $ x^{k+1} = \pi_{X_{\xi_k}}(x^k).$
As we have shown in \eqref{cfp_qg}, we have $ \mu \sigma_{F,\mu} = \sigma_X$
and by Corollary \ref{corr_linconv} we recover the widely known linear convergence rate (see \cite{Ned:10,NecRic:18}):
$$ \mathbb{E} \left[\text{dist}_{X^*}^2(x^{k})\right] 
\le (1 - \sigma_X)^k \mathbb{E} \left[\text{dist}_{X^*}^2(x^0)\right].
$$

\begin{remark}
Although the linear convergence from Corrolary \ref{corr_linconv} is fairly known, this section prove that our analysis recover standard linear convergence results from the literature, 
on particular CFPs or interpolation models.
\end{remark}	

\section{Numerical experiment}

Let the linearly constrained regression problem: $\min_x \left\{ \frac{1}{2m}\norm{Tx-y}^2_2 \;|\;  \text{s.t.} \;  C x \le d  \right\}$. Denoting the halspaces $H_{\xi} = \{x \in \rset^n \; | \; c_{\xi} \cdot x \le d_{\xi}\}$, we test our algorithm on the following model:
\begin{align}
\min\limits_{x \in \rset^n} \; & \; \mathbb{E}_{\xi \in \Omega_1}\left[\frac{1}{2}\|T_{\xi}x-y_{\xi}\|^2_2\right] +  \mathbb{E}_{\xi \in \Omega_2}\left[\mathbb{I}_{H_{\xi}}(x) \right]\label{eq:strongprob}
\end{align}
where $T \in \rset^{m \times n}, x \in \rset^n, C \in \rset^{p \times n}$. The objective function and constraints of \eqref{eq:strongprob} is defined by an average of $m+p$ functions over $\Omega = \{1, 2, \cdots, m+p\}$. We generate random data  using a standard normal distribution.
For $m= 32, n = 40, p = 200$, below we show the convergence curves of SPP variants using stepsizes $\{\frac{1}{k}, \frac{1}{k^{2/3}}, \frac{1}{k^{1/2}}\}$.
With identical parameters and initialization we take the average over 5 rounds of each SPP scheme.

\begin{figure*}[t]\label{fig:test}
	\centering
	\includegraphics[width=0.7\textwidth]{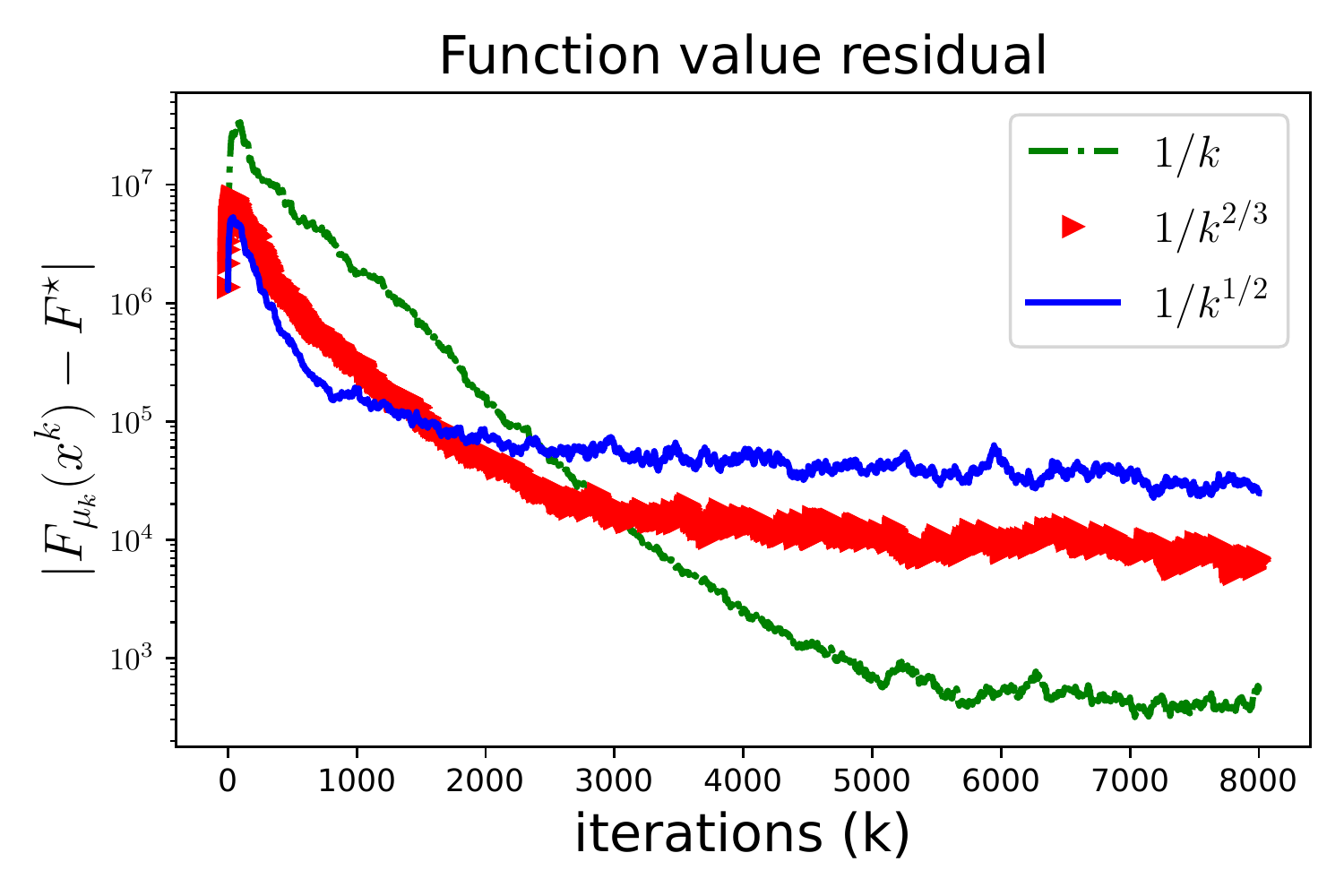}
	\includegraphics[width=0.7\textwidth]{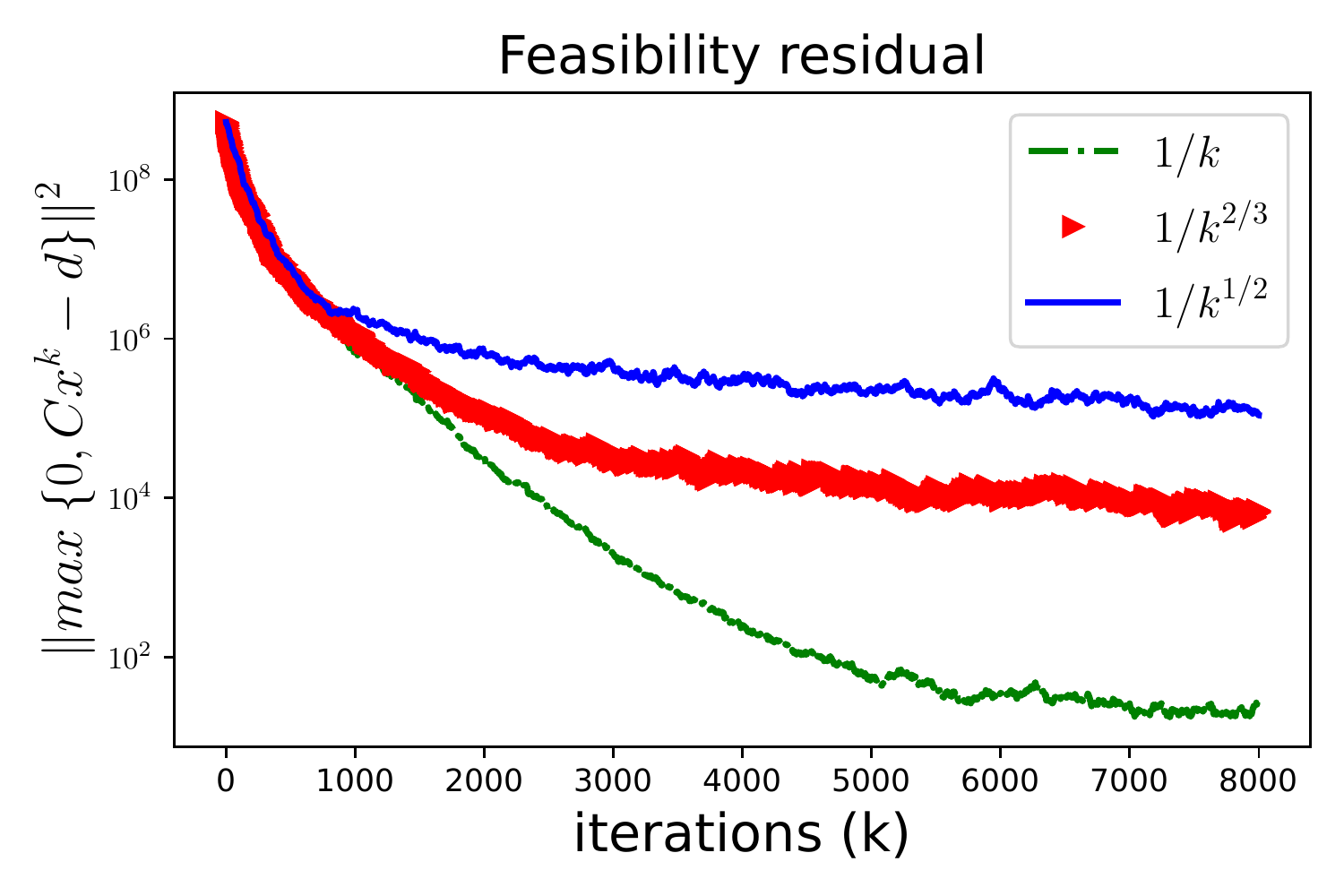}
	\caption{Given $m= 32, n = 40, p = 200$, we plotted in the top the convergence of SPP variants using stepsizes $\{\frac{1}{k}, \frac{1}{k^{2/3}}, \frac{1}{k^{1/2}}\}$, in  $|F_{\mu_k}(x^k) - F^*|$. In the bottom plot we show convergence feasibility residual.
	}
\end{figure*}

\noindent The lack of strong convexity and the presence of constraints in \eqref{eq:strongprob} leads us to examine the evolution of a measure casting together the function value residual and a feasibility penalty. Note that, for particular model \eqref{eq:strongprob}, the envelope
$$F_{\mu}(x) = \mathbb{E} \left[ \frac{(T_{\xi}x - y_{\xi})^2}{1 + \mu \norm{T_{\xi}}^2} + \frac{1}{2\mu} (\max\{0,C_{\xi}x-d_{\xi}\})^2 \right], $$
is composed of a weighted sum of quadratic components (from objective function of \eqref{eq:strongprob}) and a feasibility penalty term. A small positive $\mu$ take the first term close to objective function $\norm{Tx-y}^2$.
Thus, in the top of Figure \ref{fig:test} we plot the convergence of  $|F_{\mu_k}(x^k) - F^*|$ for the three schemes. We use CVX to determine $F^\star$ with $\varepsilon=10^{-6}$. In the bottom plot we show convergence of feasibility violation residual. Notice that the SPP with $\frac{1}{k}$ stepsize has the best performance and overall, larger is the stepsize exponent, the better is the convergence rate, which represents a confirmation of the results of Theorem \ref{th_sublinearrate}.


\section*{Funding}

This work was supported by BRD Groupe Societe Generale through Data Science Research Fellowships of 2019.





\section{Appendix }

\begin{proof}[Proof of Lemma \ref{lemma_auxres}]
	It is straightforward that 
	\begin{align*}
	F_{\mu}(x;\xi) = \min\limits_{z\in \rset^n} \; F(z;\xi) + \frac{1}{2\mu}\norm{z - x}^2 \le F(x;\xi) \qquad \forall x \in \rset^n.
	\end{align*}	
	By taking expectation w.r.t. $\xi$ in both sides we get $(i)$. In order to prove $(ii)$, let $z \in \rset^n$. Then, given $x^* \in X^*$ and $g_F(x^*;\xi) \in \partial F(x^*;\xi)$, by convexity of $f(\cdot;\xi)$ we have:
	\begin{align*}
	F^* - F_{\mu}(x)  & = \mathbb{E}\left[  F(x^*;\xi) - F(z_{\mu}(x;\xi);\xi)  - \frac{1}{2\mu}\norm{z_{\mu}(x;\xi) - x}^2 \right] \nonumber\\
	& \le \mathbb{E}\left[\langle g_F(x^*;\xi), x^* - z_{\mu}(x;\xi) \rangle - \frac{1}{2\mu}\norm{z_{\mu}(x;\xi) - x}^2\right] \nonumber\\
	& \le \mathbb{E}\left[\langle g_F(x^*;\xi), x^* - x \rangle + \langle g_F(x^*;\xi), x - z_{\mu}(x;\xi) \rangle - \frac{1}{2\mu}\norm{z_{\mu}(x;\xi) - x}^2\right] \nonumber\\
	& \le \mathbb{E}\left[\langle g_F(x^*;\xi), x^* - x \rangle + \max_{z} \; \langle g_F(x^*;\xi), x - z \rangle - \frac{1}{2\mu}\norm{z - x}^2\right] \nonumber\\
	& \le \langle \mathbb{E}\left[ g_F(x^*;\xi) \right], x^* - x \rangle  + \mathbb{E} \left[  \frac{\mu}{2}\norm{g_F(x^*;\xi)}^2\right] \quad \forall x^* \in X^*,
	\end{align*}
	where we recall that, based on Assumption \ref{assump_basic}$(iv)$, $\mathbb{E}\left[ g_F(x^*;\xi) \right] = 0 $. Therefore, we finally obtain
	\begin{align*}
	F^* - F_{\mu}(x)  
	& \le \frac{\mu}{2}\mathbb{E}\left[ \norm{g_F(x^*;\xi)}^2\right] \le \mathcal{S}^*_F.	
	\end{align*}	
	which confirms result $(ii)$.
	For the third part $(iii)$, denote $D_{\mu}(x) := \mathbb{E}\left[ \frac{1}{2\mu}\text{dist}_{X_{\xi}}^2(x;\xi) \right]$. Assumption \ref{assump_basic} $(iv)$ imply that each $g_f(x^*) \in \partial f(x^*)$ has a representation 
	\begin{align}\label{stoch_repr}
	g_f(x^*) = \mathbb{E}\left[ g_f(x^*;\xi) \right]
	\end{align}
	for some $g_f(x^*;\xi) \in \partial f(x^*;\xi)$. Thus without losing generality we are able to consider that: $\langle g_f(x^*), z - x^*\rangle \ge 0 $ for all $z \in X$. Then we derive that:
	\begin{align}
	& F_{\mu}(x)  - F(x^*) =  f_{\mu}(x) - f(x^*) + D_{\mu}(x) \nonumber\\
	&\ge \mathbb{E}\left[\langle g_f(x^*;\xi), z_{\mu}(x;\xi) - x^* \rangle + \frac{1}{2\mu}  \norm{z_{\mu}(x;\xi) - x}^2 \right] + D_{\mu}(x) \nonumber \\
	& \ge \mathbb{E}\left[\langle g_f(x^*;\xi), z_{\mu}(x;\xi) - x \rangle + \frac{1}{2\mu}  \norm{z_{\mu}(x;\xi) - x}^2 \right] + \mathbb{E} \left[ \langle g_f(x^*;\xi), x - x^* \rangle \right] + D_{\mu}(x)\nonumber \\
	& \overset{\eqref{stoch_repr}}{\ge} - \mathbb{E}\left[\frac{\mu}{2}  \norm{g_f(x^*;\xi)}^2 \right] +  \langle g_f(x^*), x - x^* \rangle + D_{\mu}(x)\nonumber  \\
	& \ge - \mathbb{E}\left[\frac{\mu}{2}  \norm{g_f(x^*;\xi)}^2 \right] +  \langle g_f(x^*), \pi_{X}(x) - x^* \rangle + \langle g_f(x^*), x - \pi_{X}(x) \rangle + D_{\mu}(x)\nonumber  \\
	& \ge - \mathbb{E}\left[\frac{\mu}{2}  \norm{g_f(x^*;\xi)}^2 \right] +   \langle g_f(x^*), x - \pi_{X}(x) \rangle + D_{\mu}(x)  \nonumber  \\
	& \ge - \mathbb{E}\left[\frac{\mu}{2}  \norm{g_f(x^*;\xi)}^2 \right] -  \norm{g_f(x^*)} \text{dist}_{X}(x) + \frac{\sigma_X}{2\mu}\text{dist}_{X}^2(x) \nonumber  \\
	& \ge - \mathbb{E}\left[\frac{\mu}{2}  \norm{g_f(x^*;\xi)}^2 \right] -  \frac{\mu}{2\sigma_X}\norm{g_f(x^*)}^2. \label{aprox_optimmodif1}
	\end{align}
\end{proof}

\begin{proof}[Proof of Theorem \ref{lemma_qg_proxqg}]
We make two central observations. Similarly, as in the proof of Lemma \ref{lemma_auxres}, assume w.l.g. that: $\langle g_f(x^*), z - x^*\rangle \ge 0 $ for all $z \in X$. First, using the linear regularity of the feasible set, it can be easily seen that:
\begin{align}
& F_{\mu}(x) = f_{\mu}(x) + \mathbb{E} \left[  \frac{1}{2\mu}\text{dist}_{X_{\xi}}^2(x) \right] \nonumber\\
&\ge \mathbb{E}\left[ f(x^*;\xi) + \langle g_f(x^*;\xi), z_{\mu}(x;\xi) - x^* \rangle + \frac{1}{2\mu}\norm{z_{\mu}(x;\xi) - x}^2 \right] 
	+ \frac{\sigma_X}{2\mu} \text{dist}_X^2(x)\nonumber \\
	& =  f^* + \mathbb{E}\left[ \langle g_f(x^*;\xi), z_{\mu}(x;\xi) - x \rangle + \frac{1}{2\mu}\norm{z_{\mu}(x;\xi) - x}^2 \right] \nonumber\\
	& \hspace{2cm} + \langle g_f(x^*), \pi_X(x) - x^* \rangle + \langle g_f(x^*), x - \pi_X(x) \rangle + \frac{\sigma_X}{2\mu} \text{dist}_X^2(x)\nonumber \\
	& \ge  f^* - \mu \mathbb{E}\left[ \norm{g_f(x^*;\xi)}^2 \right] + \frac{1}{4\mu}\mathbb{E}\left[\norm{z_{\mu}(x;\xi) - x}^2 \right] \nonumber\\
	& \hspace{3cm} - \norm{g_f(x^*)} \text{dist}_X(x) + \frac{\sigma_X}{2\mu} \text{dist}_X^2(x)\nonumber \\
	& \ge  f^* - \mu \mathbb{E}\left[ \norm{g_f(x^*;\xi)}^2 \right] + \frac{\mu}{4}\mathbb{E}\left[\norm{\nabla f_{\mu}(x;\xi)}^2 \right] - \frac{\mu}{\sigma_X}\norm{g_f(x^*)}^2 + \frac{\sigma_X}{4\mu} \text{dist}_X^2(x),  
	\label{lowbound_moreau}
	\end{align}
	for all $x \in \rset^n$.
	In the first inequality we used convexity of $f(\cdot;\xi)$
	and in the third the inequality $ab \le \frac{a}{2\alpha} + \frac{\alpha b}{2}$, for $a = \norm{g_f(x^*)},b = \text{dist}_X(x), \alpha = \frac{\sigma_X}{2\mu}$. Now we derive two auxiliary inequalities, useful for the final constant bounds. Since $f$ is differentiable then $g_f(x) = \mathbb{E}[g_f(x;\xi)]$.
	First, using the smoothing gradient inequality from Lemma \ref{lemma_gradorder} we obtain:
	\begin{align}
	\frac{\mu}{4}\mathbb{E}[ & \norm{\nabla f_{\mu}(x;\xi)}^2] 
	\ge \frac{\mu}{8}\mathbb{E}\left[\norm{\nabla f_{\mu}(\pi_X(x);\xi)}^2\right] - \frac{\mu}{4} \mathbb{E}\left[\norm{\nabla f_{\mu}(x;\xi)-\nabla f_{\mu}(\pi_X(x);\xi)}^2 \right] \nonumber\\
	&\ge \frac{\mu}{8}\mathbb{E}\left[\norm{\nabla f_{\mu}(\pi_X(x);\xi)}^2\right] - \frac{1}{4\mu} \text{dist}_X^2(x) \nonumber\\
	&\overset{\text{Lemma} \; \ref{lemma_gradorder}}{\ge} \frac{\mu}{8}\mathbb{E}\left[\frac{\norm{g_f(\pi_X(x);\xi)}^2}{(1 + \mu L_{\xi})^2}\right] - \frac{1}{4\mu} \text{dist}_X^2(x) \nonumber\\
	&\overset{\text{Jensen} \; + \eqref{grad_qg}}{\ge} \frac{\sigma_f \mu}{16(1 + \mu L_{\max})^2} \text{dist}_{X^*}^2(\pi_X(x)) - \frac{1}{4\mu} \text{dist}_X^2(x) \nonumber\\
	& \ge \frac{\sigma_f\mu}{32(1 + \mu L_{\max})^2} \text{dist}_{X^*}^2(x) - \left(\frac{1}{4\mu} + \frac{\sigma_f \mu}{16(1 + \mu L_{\max})^2} \right) \text{dist}_X^2(x). \label{lastgradbound}
	\end{align}
	Second, based on similar lines as in the proof of Lemma \ref{lemma_auxres}$(ii)$, notice that:
	\begin{align}
	f_{\mu}(x) 
	& \ge f^* + \mathbb{E}\left[ \langle g_f (x^*;\xi), z_{\mu}(x;\xi) - x^* \rangle  + \frac{1}{2\mu}\norm{z_{\mu}(x;\xi) - x}^2 \right] \nonumber \\
	& \ge f^* - \frac{\mu}{2}\mathbb{E}\left[ \norm{g_f(x^*;\xi)}^2 \right] + \langle g_f(x^*), x - \pi_X(x) \rangle + \langle g_f(x^*), \pi_X(x) -x^* \rangle \nonumber\\
	& \overset{C.S.}{\ge} f^* - \frac{\mu}{2}\mathbb{E}\left[ \norm{g_f(x^*;\xi)}^2 \right] - \norm{g_f(x^*)}\text{dist}_X(x), \label{resid_lowbnd}
	\end{align}
	where in the last inequality we used Cauchy-Schwartz inequality and the first order optimality conditions.
	By combining \eqref{lowbound_moreau}-\eqref{lastgradbound}-\eqref{resid_lowbnd}, then we have:
	\begin{align}
	& F_{\mu}(x) - F^*
	\ge  \frac{\sigma_f\mu}{32(1 + \mu L_{\max})^2} \text{dist}_{X^*}^2(x) - \mu \mathbb{E}\left[ \norm{g_f(x^*;\xi)}^2 \right]  - \frac{\mu}{\sigma_X}\norm{g_f(x^*)}^2 \nonumber\\
	& \hspace{2cm}  + \frac{\sigma_X}{4\mu}\text{dist}_X^2(x)  - \left[ \frac{1}{4\mu} + \frac{\sigma_f \mu}{16(1 + \mu L_{\max})^2} \right] \text{dist}_X^2(x) \nonumber \\
	& \overset{l.r.}{\ge}  \frac{\sigma_f\mu}{32(1 + \mu L_{\max})^2} \text{dist}_{X^*}^2(x) - \mu \mathbb{E}\left[ \norm{g_f(x^*;\xi)}^2 \right] - \frac{\mu}{\sigma_X}\norm{g_f(x^*)}^2 \nonumber\\
	& \hspace{1cm} + \frac{\sigma_X}{4\mu}\text{dist}_X^2(x) - \left[ \frac{1}{2\sigma_X} + \frac{\sigma_f\mu^2}{8\sigma_X(1 + \mu L_{\max})^2} \right] (F_{\mu}(x) - F^* - (f_{\mu}(x) - F^*)) \nonumber \\
	& \overset{\eqref{resid_lowbnd}}{\ge}  \frac{\sigma_f\mu}{32(1 + \mu L_{\max})^2} \text{dist}_{X^*}^2(x) - \mu \mathbb{E}\left[ \norm{g_f(x^*;\xi)}^2 \right] - \frac{\mu}{\sigma_X}\norm{g_f(x^*)}^2 \nonumber\\
	& - \left[ \frac{1}{2\sigma_X} + \frac{\sigma_f\mu^2}{8\sigma_X(1 + \mu L_{\max})^2} \right] \bigg (  F_{\mu}(x) - F^* \nonumber \\ 
	& \hspace{3cm} - \frac{\mu}{2}\mathbb{E}\left[ \norm{g_f(x^*;\xi)}^2 \right] - \norm{g_f(x^*)}\text{dist}_X(x) \bigg ) + \frac{\sigma_X}{4\mu}\text{dist}_X^2(x), \nonumber 
	\end{align}
	where in the second inequality we used linear regularity.
	By transferring all the terms containing $F_{\mu}$ in the left hand side and denoting $c = \frac{1}{2\sigma_X} + \frac{\sigma_f\mu^2}{8\sigma_X(1 + \mu L_{\max})^2}$, then we 
	finally obtain:
	\begin{align}
	(1 + c)(F_{\mu}(x) - F^*)
	&\ge  \frac{\sigma_f\mu}{32(1 + \mu L_{\max})^2} \text{dist}_{X^*}^2(x) \nonumber \\
	& \hspace{1cm} - (1+c)\mu \mathbb{E}\left[ \norm{g_f(x^*;\xi)}^2 \right] - (1+c^2)\frac{\mu}{\sigma_X}\norm{g_f(x^*)}^2,
	\end{align}
	which immediately confirms our above result.
\end{proof}

\begin{lemma}\label{lemma_lipgrad}
	Let $h:\rset^n \to \rset^{}$ be convex and having Lipschitz continuous gradient with constant $L_h$, then the following relation hold:
	\begin{align*}
	\norm{\nabla h(x)} \le 2L_h(h(x) - h(x^*)) \qquad \forall x \in \rset^n, 
	\end{align*}
	where $x^* \in \arg\min_x h(x)$.
\end{lemma}
\begin{proof}
	From Lipschitz continuity, we have:
	\begin{align*}
	h(y) \le h(x) + \langle \nabla h(x), y-x\rangle + \frac{L_h}{2}\norm{x-y}^2.
	\end{align*}
	By minimizing both sides over $y$, we obtain:
	\begin{align*}
	h(x^*) \le h(x) - \frac{1}{2L_h}\norm{\nabla h(x)}^2,
	\end{align*}
	which confirms the result.
\end{proof}

\begin{lemma}\label{lemma_sc_link}
	Let $f$ be continuously differentiable, then $f$ is $M -$restricted strongly convex if and only if:
	\begin{align}\label{gradvarbound}
	\langle \nabla f(x) - \nabla f(y), x - y \rangle \ge \langle x - y, M (x-y)\rangle \qquad \forall x,y \in \rset^n.
	\end{align}
\end{lemma}
\begin{proof}
	Assume that $f$ is $M-$restricted strongly convex, then by adding the relation 
	\begin{align*}
	f(x) \ge f(y) + \langle \nabla f(y),x - y\rangle + \frac{1}{2}\langle x - y, M (x-y) \end{align*}
	with the same but with interchanged $x$ and $y$ then we obtain the first implication. Next, assume that \eqref{gradvarbound} holds.
	By the Mean Value Theorem we have:
	\begin{align*}
	f(x) 
	& = f(y) + \int\limits_{0}^1 \langle\nabla f(\tau x + (1-\tau)y) , x- y\rangle d\tau \\
	& = f(y) + \langle \nabla f(y) , x - y\rangle +  \int\limits_{0}^1 \frac{1}{\tau} \langle\nabla f(\tau x + (1-\tau)y) - \nabla f(y) , \tau( x- y )\rangle d\tau \\
	& \overset{\eqref{gradvarbound}}{\ge} f(y) + \langle \nabla f(y) , x - y\rangle +  \int\limits_{0}^1 \frac{\tau}{2} \langle x- y, M ( x- y )\rangle d\tau \\
	& = f(y) + \langle \nabla f(y) , x - y\rangle +  \frac{1}{2}\langle x- y, M ( x- y )\rangle d\tau,
	\end{align*}
	which confirms the second implication.
\end{proof}

\begin{proof}[proof of Lemma \ref{lemma_restrictedscqg}]
	From the $M_{\xi}-$restricted strong convexity assumption we have:
	\begin{align*}
	\langle \nabla f(x;\xi) - \nabla f(y;\xi), x - y\rangle \ge \norm{x - y}^2_{M_{\xi}}.
	\end{align*}
	By taking $x = z_{\mu}(x;\xi)$ and $y = z_{\mu}(y;\xi)$ then the above relation implies:
	\begin{align}
	&\norm{z_{\mu}(x;\xi) - z_{\mu}(y;\xi)}^2_{M_{\xi}} 
	 \le \langle \nabla f(z_{\mu}(x;\xi);\xi) - \nabla f(z_{\mu}(y;\xi);\xi), z_{\mu}(x;\xi) - z_{\mu}(y;\xi)\rangle \nonumber\\
	& \le \frac{1}{\mu}\langle x - z_{\mu}(x;\xi) - (y - z_{\mu}(y;\xi)), z_{\mu}(x;\xi) - z_{\mu}(y;\xi)\rangle \nonumber\\
	& \le \frac{1}{\mu} \langle x  - y , z_{\mu}(x;\xi) - z_{\mu}(y;\xi)\rangle - \frac{1}{\mu}\norm{z_{\mu}(x;\xi) - z_{\mu}(y;\xi)}^2.
	\label{prepseudocontraction}
	\end{align}
	After simple manipulations, using the Cauchy-Schwartz inequality the last inequality \eqref{prepseudocontraction} further implies:
	\begin{align}\label{precontraction}
	& \langle  z_{\mu}(x;\xi) - z_{\mu}(y;\xi),  (I_n + \mu M_{\xi})(z_{\mu}(x;\xi) - z_{\mu}(y;\xi))\rangle 
	\le \langle x  - y , z_{\mu}(x;\xi) - z_{\mu}(y;\xi)\rangle \nonumber \\
	&  =  \langle (I_n + \mu M_{\xi})^{-1/2}(x  - y) ,  (I_n + \mu M_{\xi})^{1/2}(z_{\mu}(x;\xi) - z_{\mu}(y;\xi))\rangle \nonumber \\
	&\overset{C.-S.}{\le} \norm{(I_n + \mu M_{\xi})^{-1/2}(x  - y)} \norm{(I_n + \mu M_{\xi})^{1/2}( z_{\mu}(x;\xi) - z_{\mu}(y;\xi))}.
	\end{align}
	An important consequence of \eqref{precontraction} is the following contraction property:
	\begin{align}\label{contraction}
	\norm{(I_n + \mu M_{\xi})^{1/2}( z_{\mu}(x;\xi) - z_{\mu}(y;\xi))} \le \norm{(I_n + \mu M_{\xi})^{-1/2}(x  - y)},
	\end{align}
	for all $x,y \in \rset^n, \xi \in \Omega$. Now by using the particular structure of $\nabla f_{\mu}(\cdot;\xi)$ and that fact that $I_n + \mu M_{\xi}$ is invertible, we have:
	\begin{align*}
	& \langle \nabla f_{\mu}(x;\xi) - \nabla f_{\mu}(y;\xi), x - y \rangle
	=  \frac{1}{\mu}\norm{x-y}^2 - \frac{1}{\mu}\langle z_{\mu}(x;\xi) - z_{\mu}(y;\xi), x-y\rangle \\
	& =  \frac{1}{\mu}\norm{x-y}^2 - \frac{1}{\mu} \langle (I_n + \mu M_{\xi})^{1/2} \left(z_{\mu}(x;\xi) - z_{\mu}(y;\xi)\right), (I_n + \mu M_{\xi})^{-1/2}(x-y)\rangle.	
	\end{align*}
	By taking expectation in both sides and also using the Cauchy-Schwartz inequality and the contraction property \eqref{contraction} we get:
	\begin{align}
	& \langle \nabla F_{\mu}(x) -  \nabla F_{\mu}(y),  x - y \rangle \nonumber \\
	& =  \frac{1}{\mu}\norm{x-y}^2 - \frac{1}{\mu} \mathbb{E}\left[ \langle (I_n + \mu M_{\xi})^{1/2} \left(z_{\mu,\xi}(x) - z_{\mu,\xi}(y)\right), (I_n + \mu M_{\xi})^{-1/2}(x-y)\rangle \right] \nonumber \\	
	& \overset{C.S.}{\ge}  \frac{1}{\mu}\norm{x-y}^2 - \frac{1}{\mu} \mathbb{E}\left[ \norm{ (I_n + \mu M_{\xi})^{1/2} \left(z_{\mu}(x;\xi) - z_{\mu}(y;\xi)\right)} \norm{(I_n + \mu M_{\xi})^{-1/2}(x-y)} \right] \nonumber \\
	&  \overset{\eqref{contraction}}{\ge}  \frac{1}{\mu}\norm{x-y}^2 - \frac{1}{\mu} \mathbb{E}\left[\norm{(I_n + \mu M_{\xi})^{-1/2}(x-y)}^2 \right] \nonumber \\	 
	&  =  \frac{1}{\mu}\norm{x-y}^2 - \frac{1}{\mu}  \langle x - y, \mathbb{E}\left[(I_n + \mu M_{\xi})^{-1}\right] (x-y)\rangle \nonumber \\
	&  =  \frac{1}{\mu}  \langle x - y, I - \mathbb{E}\left[(I_n + \mu M_{\xi})^{-1}\right] (x-y) \rangle.	\label{sc_moreau} 
	\end{align}
	We further deduce that:
	\begin{align}
	I_n & - ( I_n + \mu M_{\xi})^{-1} 
	 =  \left[I_n - ( I_n + \mu M_{\xi})^{-1} \right]^{1/2} \left[ I_n - ( I_n + \mu M_{\xi})^{-1}\right]^{1/2} \nonumber \\
	& =  \left[I_n + \mu M_{\xi} - I_n \right]^{1/2}( I_n + \mu M_{\xi})^{-1/2} ( I_n + \mu M_{\xi})^{-1/2} \left[  I_n + \mu M_{\xi} - I_n \right]^{1/2} \nonumber \\
	& = \mu M_{\xi}^{1/2} (I_n + \mu M_{\xi})^{-1} M_{\xi}^{1/2}.\label{sc_morconst_bound}
	\end{align}
	By using this bound into \eqref{sc_moreau}, then we finally obtain the strong convexity relation:
	\begin{align}
	\langle \nabla F_{\mu}(x) - \nabla F_{\mu}(y),  x - y \rangle \ge 
	&\frac{1}{\mu}  \langle x - y, I_n - \mathbb{E}\left[(I_n + \mu M_{\xi})^{-1}\right] (x-y) \rangle \nonumber \\
	& \overset{\eqref{sc_morconst_bound}}{\ge} \langle x - y, \mathbb{E}\left[ M_{\xi}^{1/2} (I_n + \mu M_{\xi})^{-1} M_{\xi}^{1/2}\right] (x-y) \rangle \nonumber \\
	& \ge  \langle x - y, \mathbb{E}\left[ M_{\xi}\lambda_{\min}\left((I_n + \mu M_{\xi})^{-1}\right) \right] (x-y) \rangle. \label{sc_morfinal}
	\end{align}
	As the last step of the proof, by observing $\lambda_{\min}\left((I_n + \mu M_{\xi})^{-1}\right) = \frac{1}{\lambda_{\max}(I_n + \mu M_{\xi})}$ and by applying Lemma \ref{lemma_sc_link} with $f = F_{\mu}$ and $M = \mathbb{E}\left[ M_{\xi}\lambda_{\min}\left((I_n + \mu M_{\xi})^{-1}\right) \right]$, makes the connection between \eqref{sc_morfinal} and the above result. 
\end{proof}

\begin{proof}[Proof of Theorem \ref{th_rsc_proxqg}]
	Recall that $F_{\mu}(x) = f_{\mu}(x) + D_{\mu}(x),$ where $D_{\mu}(x) : =  \mathbb{E}[\text{dist}^2_{X_{\xi}}(x)]$. Let $x^*_{\mu} \in \arg\min_x F_{\mu}(x), \tilde{x}^*_{\mu} = \arg\min\limits_{x \in X} f_{\mu}(x)$ and denote $\hat{M}_f = \mathbb{E}\left[ \frac{M_{\xi}}{\lambda_{\max}(I_n + \mu M_{\xi})}\right]$. Then, by using Lemma \ref{lemma_restrictedscqg} for $F_{\mu}$, we obtain:
	\begin{align}
	& F_{\mu}(x)  = f_{\mu}(x) + D_{\mu}(x) \nonumber\\
	& \ge f_{\mu}(\tilde{x}^*_{\mu}) + \langle \nabla f_{\mu}(\tilde{x}^*_{\mu}), x - \tilde{x}^*_{\mu} \rangle + \frac{1}{2} \langle x - \tilde{x}^*_{\mu},\hat{M}_f(x - \tilde{x}^*_{\mu})\rangle + D_{\mu}(x) \nonumber \\
	& = F_{\mu}(\tilde{x}^*_{\mu}) + \langle \nabla f_{\mu}(\tilde{x}^*_{\mu}), \pi_X{(x)} - \tilde{x}^*_{\mu} \rangle + \langle \nabla f_{\mu}(\tilde{x}^*_{\mu}), x - \pi_X(x) \rangle + \nonumber \\
	&  \hspace{15pt} \frac{1}{2} \langle x - \tilde{x}^*_{\mu},\hat{M}_f(x - \tilde{x}^*_{\mu})\rangle + D_{\mu}(x) \nonumber \\
	& \overset{C.S.}{\ge} F_{\mu}(\tilde{x}^*_{\mu})  - \norm{\nabla f_{\mu}(\tilde{x}^*_{\mu})} \text{dist}_X(x) +  \frac{1}{2} \langle x - \tilde{x}^*_{\mu},\hat{M}_f(x - \tilde{x}^*_{\mu})\rangle + D_{\mu}(x) \nonumber \\ 
	& \overset{l.r.}{\ge} F_{\mu}(\tilde{x}^*_{\mu})  - \norm{\nabla f_{\mu}(\tilde{x}^*_{\mu})} \text{dist}_X(x) +  \frac{1}{2} \langle x - \tilde{x}^*_{\mu},\hat{M}_f(x - \tilde{x}^*_{\mu})\rangle + \frac{1}{2}D_{\mu}(x) + \frac{\sigma_X}{4\mu}\text{dist}_X^2(x) \nonumber \\ 
	& \ge F_{\mu}(\tilde{x}^*_{\mu})  - \frac{\mu}{\sigma_X} \norm{\nabla f_{\mu}(\tilde{x}^*_{\mu})}^2  +  \frac{1}{2} \langle x - \tilde{x}^*_{\mu},\hat{M}_f(x - \tilde{x}^*_{\mu})\rangle  + \frac{1}{2}D_{\mu}(x), \label{pqg_prelim}
	\end{align}
	where in the second inequality we used Cauchy-Schwartz inequality and in the third we used linear regularity of feasible sets $\{X_{\xi}\}_{\xi \in \Omega}$.
	To bound further the right hand side, we first have from Lemma \ref{lemma_lipgrad}:
	\begin{align}\label{smoothgrad_bound}
	& \norm{\nabla f_{\mu}(\tilde{x}^*_{\mu})}^2
	\le \frac{2}{\mu} \left(f_{\mu}(\tilde{x}^*_{\mu}) - f_{\mu}^* \right) 
	\le \frac{2}{\mu} \left(f_{\mu}(x^*) - f_{\mu}^* \right) \nonumber\\
	&\overset{\text{Lemma} \; \ref{lemma_auxres} \; (i)}{\le} \frac{2}{\mu} \left(f(x^*) - f_{\mu}^* \right) \overset{\eqref{resid_lowbnd}}{\le} \mathbb{E}\left[\norm{g_f(x^*;\xi)}^2 \right] + \frac{2}{\mu} \norm{g_f(x^*)}\underbrace{\text{dist}_X(\tilde{x}^*_{\mu})}_{=0},
	\end{align}
	where in the last inequality we applied Lemma \ref{lemma_auxres}$(ii)$ with $F=f$. Second, by applying  \eqref{pqg_prelim} with $x = x^*$ and by using \eqref{smoothgrad_bound}, we obtain:
	\begin{align}
	\frac{1}{2} \langle x^* - \tilde{x}^*_{\mu}, & \hat{M}_f(x^* - \tilde{x}^*_{\mu})\rangle 
	\le F_{\mu}(x^*) - F_{\mu}(\tilde{x}^*_{\mu}) + \frac{\mu}{\sigma_X}\mathbb{E}\left[\norm{g_f(x^*;\xi)}^2 \right] \nonumber\\
	& \le F(x^*) - F_{\mu}(\tilde{x}^*_{\mu}) + \frac{\mu}{\sigma_X}\mathbb{E}\left[\norm{g_f(x^*;\xi)}^2 \right] \nonumber\\
	& \overset{\text{Lemma} \; \ref{lemma_auxres} \; (iii)}\le \frac{\mu}{2} \left(1 + \frac{2}{\sigma_X} \right)\mathbb{E}\left[\norm{g_f(x^*;\xi)}^2 \right]  + \frac{\mu}{2\sigma_X}\norm{g_f(x^*)}^2  \label{distoptims_bound}
	\end{align}
	Combining the upper bounds \eqref{smoothgrad_bound}-\eqref{distoptims_bound} into relation \eqref{pqg_prelim}, we derive:
	\begin{align}
	& f_{\mu}(x) 
	\ge f_{\mu}(\tilde{x}^*_{\mu})  - \frac{\mu}{\sigma_X} \norm{\nabla f_{\mu}(\tilde{x}^*_{\mu})}^2  -  \frac{1}{2} \langle x^* - \tilde{x}^*_{\mu},\hat{M}_f(x^* - \tilde{x}^*_{\mu})\rangle + \nonumber\\
	& \hspace{2cm} \frac{1}{4} \langle x - x^*,\hat{M}_f(x - x^*)\rangle + \frac{1}{2}D_{\mu}(x) \nonumber \\
	& \overset{\eqref{smoothgrad_bound}-\eqref{distoptims_bound}}{\ge} F_{\mu}(\tilde{x}^*_{\mu})  -  \frac{\mu}{2} \left(1 + \frac{2}{\sigma_X} \right)\mathbb{E}\left[\norm{g_f(x^*;\xi)}^2 \right]  - \frac{\mu}{2\sigma_X}\norm{g_f(x^*)}^2 \nonumber \\      
	& \hspace{2cm} + \frac{1}{4} \langle x - x^*,\hat{M}_f(x - x^*)\rangle + \frac{1}{2}D_{\mu}(x) \nonumber \\
	& \overset{l.r.}{\ge} F_{\mu}(\tilde{x}^*_{\mu})  -  \frac{\mu}{2} \left(1 + \frac{2}{\sigma_X} \right)\mathbb{E}\left[\norm{g_f(x^*;\xi)}^2 \right]  - \frac{\mu}{2\sigma_X} \norm{g_f(x^*)}^2 \nonumber \\      
	& \hspace{2cm} + \frac{1}{4} \langle x - x^*,\hat{M}_f(x - x^*)\rangle + \frac{1}{2}D_{\mu}(x). \nonumber
	\end{align}
	Lastly, by taking into account that 
	$X^* = \{x^*: M_f^{1/2}x^* = y^*, x^* \in X_{\xi}, \xi \in \Omega \}$, then:
	\begin{align}
	&\norm{\hat{M}_f^{1/2}(x - x^*)}^2 + D_{\mu}(x) \nonumber\\
	&= \frac{1}{1 + \mu \lambda_{\max}(M_{\max})} \bigg( \langle x-x^*, \left[1 + \mu \lambda_{\max}(M_{\max})\right]\hat{M}_f(x - x^*) \rangle \nonumber\\
	& \hspace{5cm} + \frac{1 + \mu \lambda_{\max}(M_{\max})}{2\mu} \mathbb{E}[\text{dist}_{X_{\xi}}^2(x)] \bigg) \nonumber\\
	& \ge \frac{1}{1 + \mu \lambda_{\max}(M_{\max})} \bigg( \langle x-x^*, \mathbb{E}\left[\frac{1 + \mu \lambda_{\max}(M_{\max}) M_{\xi}}{\lambda_{\max}(I_n + \mu M_{\xi})} \right] (x - x^*) \rangle \nonumber\\
	& \hspace{5cm}  + \frac{\lambda_{\max}(M_{\max})}{2} \mathbb{E}[\text{dist}_{X_{\xi}}^2(x)]\bigg) \nonumber\\
	& \ge \frac{\min\{1, \lambda_{\max}(M_{\max})\}}{2(1 + \mu \lambda_{\max}(M_{\max}))} \left( \langle x-x^*, M_f (x - x^*) \rangle + \mathbb{E}[\text{dist}_{X_{\xi}}^2(x)]\right) \nonumber \\
	& \ge \frac{\min\{1, \lambda_{\max}(M_{\max})\}}{2(1 + \mu \lambda_{\max}(M_{\max}))} \left( \norm{M_f^{1/2}x - y^*}^2 + \mathbb{E}[\text{dist}_{X_{\xi}}^2(x)]\right) \nonumber \\
	& \overset{l.r.}{\ge} \frac{\min\{1, \lambda_{\max}(M_{\max})\}}{2(1 + \mu \lambda_{\max}(M_{\max}))} \sigma_{f,X} \text{dist}_{X^*}^2(x), \quad \forall x \in \rset^n, \label{hoffbnd_convex}
	\end{align}
	where in the last inequality we have used that $X^* = \hat{X} \cap X$ and the linear regularity. This last argument leads to the final lower bound:
	\begin{align*}
	& F_{\mu}(x) 
	\ge  F_{\mu}(\tilde{x}^*_{\mu})  -  \frac{\mu}{2} \left(1 + \frac{2}{\sigma_X} \right)\mathbb{E}\left[\norm{g_f(x^*;\xi)}^2 \right]  - \frac{\mu}{2\sigma_X} \norm{g_f(x^*)}^2     \nonumber\\ 
	& \hspace{5cm} +  \frac{1}{4} \langle x - x^*,\hat{M}_f(x - x^*)\rangle + \frac{1}{2}D_{\mu}(x) \nonumber \\
	& \overset{\eqref{hoffbnd_convex}}{\ge}
	F_{\mu}(\tilde{x}^*_{\mu})  -  \frac{\mu}{2} \left(1 + \frac{2}{\sigma_X} \right)\mathbb{E}\left[\norm{g_f(x^*;\xi)}^2 \right]  - \frac{\mu}{\sigma_X}\left(\frac{1}{2} + \frac{2}{\sigma_X^2}\right) \norm{g_f(x^*)}^2   \nonumber \\ 
	& \hspace{5cm} +  \frac{\min\{1, \lambda_{\max}(M_{\max})\}}{4(1 + \mu \lambda_{\max}(M_{\max}))} \sigma_{f,X} \text{dist}_{X^*}^2(x), 
	\end{align*}
	which confirms the constants from part $(i)$.
	For $(ii)$ and $(iii)$ we commonly derive
	\begin{align}
	F_{\mu}(x) & = f_{\mu}(x) + D_{\mu}(x) \nonumber\\
	& \ge F_{\mu}(x^*_{\mu}) + \langle \nabla F_{\mu}(x^*_{\mu}), x - x^*_{\mu} \rangle + \frac{1}{2} \langle x - x^*_{\mu},\hat{M}_f(x - x^*_{\mu})\rangle \nonumber\\
	& \ge F_{\mu}(x^*_{\mu}) - \frac{1}{2} \langle x^* - x^*_{\mu},\hat{M}_f(x^* - x^*_{\mu})\rangle + \frac{1}{4} \langle x - x^*,\hat{M}_f(x - x^*)\rangle \label{pqg_prelim_sc}
	\end{align}
	On the other hand, by taking $x = x^* $ in \eqref{pqg_prelim_sc},  we get
	\begin{align}
	\frac{1}{2}\langle x^* - x^*_{\mu}, & \hat{M}_f(x^* - x^*_{\mu})\rangle 
	\le F_{\mu}(x^*) - F_{\mu}(x^*_{\mu}) \nonumber \\
	&\overset{Lemma \;\ref{lemma_auxres} \; (i)}{\le} F^* - F_{\mu}(x^*_{\mu})  \overset{Lemma \; \ref{lemma_auxres} \; (ii)}{\le} \frac{\mu}{2}\mathbb{E}[\norm{g_F(x^*;\xi)}^2]. \label{moreau_qg}
	\end{align}
	From \eqref{pqg_prelim_sc} and \eqref{moreau_qg} we obtain the weak linear regularity relation:
	\begin{align*}
	F_{\mu}(x) 
	& \ge F_{\mu}^* +  \frac{1}{4} \langle x - x^*,\hat{M}_f(x - x^*)\rangle - \frac{1}{2}\langle x^* - x^*_{\mu},\hat{M}_f(x^* - x^*_{\mu})\rangle  \nonumber \\
	& \ge F_{\mu}^* +  \frac{\lambda_{\min}(\hat{M}_f)}{4} \norm{x - x^*}^2 - \frac{\mu}{2}\mathbb{E}[\norm{g_F(x^*;\xi)}^2]
	\quad \forall x \in \rset^n,
	\end{align*}
	which confirms result $(ii)$. 
	Lastly if $X$ is linearly regular then, using Lemma $\ref{lemma_auxres} \; (iii)$, \eqref{moreau_qg} transforms into:
	\begin{align}
	\frac{1}{2}\langle x^* - x^*_{\mu},\hat{M}_f(x^* - x^*_{\mu})\rangle 
	&\le F^* - F_{\mu}(x^*_{\mu}) \nonumber \\
	& \overset{Lemma \; \ref{lemma_auxres} \; (iii)}{\le} \frac{\mu}{2}\mathbb{E}[\norm{g_f(x^*;\xi)}^2] + \frac{\mu}{2\sigma_X}\norm{g_f(x^*)}^2, \label{moreau_qg2}
	\end{align}
	and following the same lines as in the previous result $(ii)$, we immediately obtain the constants from $(iii)$.
\end{proof}

\end{document}